\documentclass{amsart}
\usepackage[]{latexsym,amssymb,amsmath,amsfonts, amsthm}
\usepackage[all,cmtip,ps]{xy}
\usepackage{enumitem}

\DeclareMathOperator{\R}{\mathbf{R}}

\def\D{\mathcal{D}}

\def\End{\operatorname{End}}
\def\Im{\operatorname{Im}}
\def\Hom{\operatorname{Hom}}
\def\Ext{\operatorname{Ext}}

\def\dualita#1#2{\mathrel{
                 \mathop{\vcenter{
                 \offinterlineskip
                 \hbox to 0.6truecm{\rightarrowfill}
                 \hbox to 0.6truecm{\leftarrowfill}}}%
                 \limits_{#2}^{#1}}}
\DeclareMathOperator{\Add}{Add}

\DeclareMathOperator{\Prod}{Prod}
\DeclareMathOperator{\injdim}{idim}

\DeclareMathOperator{\gen}{gen}
\DeclareMathOperator{\pres}{pres}

\DeclareMathOperator{\Cogen}{Cogen}
\DeclareMathOperator{\Ker}{Ker}
\DeclareMathOperator{\Rej}{Rej}

\DeclareMathOperator{\Coker}{Coker}

\newtheorem{theorem}{Theorem}[section]

\newtheorem{corollary}[theorem]{Corollary}

\newtheorem{definition}[theorem]{Definition}
\newtheorem{example}[theorem]{Example}

\newtheorem{lemma}[theorem]{Lemma}

\newtheorem{proposition}[theorem]{Proposition}
\theoremstyle{remark}
\newtheorem{remark}[theorem]{Remark}

\newcommand*{\rMod}{\textrm{\textup{Mod-}}}

\newcommand*{\lMod}{\textrm{\textup{-Mod}}}

\begin{document}

\title{Derived dualities induced by a 1-cotilting bimodule}
\author{Francesca Mantese, Alberto Tonolo}
\address[F. Mantese]{Dipartimento di Informatica, Universit\`a degli Studi di Verona, strada Le Grazie  15, I-37134 Verona - Italy}
\email{francesca.mantese@univr.it}
\address[A. Tonolo]{ Dipartimento di Matematica, Universit\`a degli studi di Padova, via Trieste 63, I-35121 Padova Italy}
\email{tonolo@math.unipd.it}
\thanks{Research supported by:
Progetto di Eccellenza Fondazione Cariparo "Algebraic structures and their applications: Abelian and derived categories, algebraic entropy
and representation of algebras" and 
Progetto di Ateneo ``Categorie Differenziali Graduate'' CPDA105885}
\date{\today}

\begin{abstract}
In this paper we characterize the modules and the complexes involved in  the dualities induced by a 1-cotilting bimodule in terms of a linear compactness condition. Our result generalizes the classical characterization of reflexive modules with respect to Morita dualities. The linear compactness notion considered, permits us to obtain finiteness properties of the rings and modules involved.
\end{abstract}
\maketitle

\section*{Introduction}
Let $R$ and $S$ be two arbitrary associative rings with $1\not=0$. We denote by $R\lMod$ and $\rMod S$ the categories of left $R$-modules and of right $S$-modules.
Morita and Azumaya in \cite{Mor58} and \cite{Azu59}  studied the additive category dualities between two categories of left $R$- and right $S$-modules closed under sub and factor modules, and containing all the finitely generated modules. They proved that these are equivalent to the category dualities given by the contravariant $\Hom$ functors induced by \emph{Morita bimodules}, i.e., bimodules $_RW_S$ such that $_RW$ and $W_S$ are injective cogenerators and $R=\End W_S$, $S=\End{}_RW$. The modules in the domain and in the range of these dualities are called \emph{Morita reflexive}. 

In \cite{Mul70} M\"uller proved that the Morita reflexive
modules  coincide with the \emph{linearly compact} modules.

The 1-cotilting modules generalize injective cogenerators: indeed, they are modules which are injective exactly on the subcategory they cogenerate.  
The cotilting theory studies the dualities induced by the contravariant functors $\Hom_R(-,U)$, $\Hom_S(-,U)$ and $\Ext^1_R(-,U)$, $\Ext^1_S(-,U)$ associated to a (faithfully balanced) 1-cotilting bimodule. 
In 2006 Robert Colby and Kent Fuller in their monograph \cite{ColFul0610} provided a unified approach to the various theories of equivalence and duality between categories of modules developed in the last 50 years, concentrating on those induced by tilting and cotilting modules. The rich bibliography in \cite{ColFul0610} takes into account the contribution in this field of many mathematicians.

Given a 1-cotilting bimodule $_RU_S$, the two pair of contravariant functors
\[\Hom_R(-,U),\ \Hom_S(-,U)\text{ and } \Ext^1_R(-,U),\ \Ext^1_S(-,U)\] play the same role, but their behavior  is not symmetric. This is essentially motivated by the fact that $\Ext^1(\Hom(-,U),U)$ vanishes, while the other composition $\Hom(\Ext^1(-,U),U)$ is in general different from zero. Many papers dedicated to cotilting theory have to handle with this asymmetry, which affects the study of the duality induced by the contravariant functors $\Ext^1_R(-,U)$, $\Ext^1_S(-,U)$. The problem can be bypassed restricting to noetherian rings and finitely generated modules (see e.g. \cite{Miy86}, \cite{Col89}, \cite{Ton04}), since $\Hom(\Ext^1(-,U),U)$ is zero on finitely related modules (see Lemma~\ref{lemma:Bon}).
Otherwise, one can circumscribe his interest to particular classes of modules where $\Hom(\Ext^1(-,U),U)$ is zero, as in \cite{Col00},\cite{Man01} and \cite{Ton00}. 

In this paper we have undertaken a different direction, facing this asimmetry. Recently, in \cite{ManTon10} we have studied the dualities induced by a cotilting bimodule in the framework of the derived categories of modules over an arbitrary associative ring. We think that this is the correct setting for understanding, also at the level of the module categories, the nature of these dualities. Considering the total derived functors $\R\Hom(-,{}_RU)$ and $\R\Hom(-,U_S)$ and their cohomologies is indeed possible to evaluate the interplay of the functors $\Hom_R(-,U)$, $\Hom_S(-,U)$ and $\Ext^1_R(-,U)$, $\Ext^1_S(-,U)$. In particular the vanishing of $\Hom(\Ext^1(-,U),U)$ is better understood and, naturally inserted in the notion of reflexive object, it looses its technical aspect.

The aim of this paper is to characterize in terms of a suitable notion of linear compactness the \emph{$\D$-reflexive complexes}, i.e., the complexes in the domain and in the range of the duality induced by the functors $\R\Hom(-,{}_RU)$ and $\R\Hom(-,U_S)$. In such a way we generalize to cotilting dualities between derived categories the classical M\"uller result. 
In \cite[Corollary 3.6, Example~4.6]{ManTon10} we have proved that a complex is $\D$-reflexive if and only if its cohomologies are $\D$-reflexive. Therefore for characterizing the $\D$-reflexive complexes is sufficient to characterize the $\D$-reflexive stalk complexes, i.e., the $\D$-reflexive modules considered as complexes concentrated in degree 0.

In \cite[Proposition~5.3.7]{ColFul0610}, consequence of results in \cite{Baz03} and \cite{ManRuzTon03}, a characterization of modules in the domain and in the range of the duality induced by the contravariant functors $\Hom_R(-,U)$, $\Hom_S(-,U)$ is given in terms of suitable notion of density and linear compactness. Generalizing a result of Wisbauer \cite[47.7]{Wis91}, we characterize in Theorem~\ref{thm:densita} the bimodules $_RU_S$ for which the density condition is automatically satisfied.
In Theorems~\ref{teo:fond1}, \ref{teo:fondam2} and Corollary~\ref{cor:fond} we characterize the modules in the domain and in the range of the duality induced by the contravariant functors $\Ext^1_R(-,U)$, $\Ext^1_S(-,U)$ employing a suitable notion of linear compactness. 
The linear compactness notion considered, permits us to obtain finiteness properties of the rings and modules involved in the duality.
In particular we extend the Osofsky's result on the impossibility for an infinite direct sum to be involved in a Morita duality to the cotilting case (see Proposition~\ref{prop:Osofsky}), and we prove that a right noetherian ring admitting a cotilting duality is necessarily semiperfect (see Proposition~\ref{prop:semiperfect}). Finally, summarizing what we have obtained in the previous sections, we conclude the paper with the characterization of $\D$-reflexive complexes (see Corollary~\ref{cor:complessiriflessivi}).

\section{Preliminaries and Notation}
Let $R$ and $S$ be two arbitrary associative rings with $1\not=0$. We denote by $R\lMod$ and $\rMod S$ the categories of left $R$-modules and of right $S$-modules, and by $\D(R)$ and $\D(S)$ the corresponding derived categories. 
\begin{definition}\label{definizione:1cotilting}
 A left $R$-module (resp. right $S$-module) $U$ is \textbf{1-cotilting}
 if it satisfies the following conditions: 
\begin{enumerate}
\item $\injdim {}U\leq 1$;
\item $\Ext^1(U^{\alpha}, U)=0$ for each cardinal $\alpha$;
\item if $\Hom(M,U)=0=\Ext^1(M,U)$ then $M=0$.
\end{enumerate}
 A bimodule $_RU_S$ is  \textbf{1-cotilting} if both $_RU$ and $U_S$ are 1-cotilting.
\end{definition}
It can be shown (see \cite[Proposition~2.3]{AngTonTrl01}) that property $(3)$ in Definition~\ref{definizione:1cotilting} can be replaced by
\begin{enumerate}
\item[$(3')$] \emph{denoted by $W$ an injective cogenerator, there is an exact sequence 
\[0\to U_1\to U_0\to W\to 0\]
where $U_1, U_0\in\Prod U$, the subcategory of the direct summands of direct products of copies of $U$.}
\end{enumerate}

Consider the following subcategories associated to a module $L$:
\begin{itemize}
\item $\Cogen L:=\{M\in R\lMod:M\leq L^\alpha\text{ for a suitable cardinal }\alpha\}$;
\item ${}^\perp L:=\{M\in R\lMod:\Ext^1_R(M,L)=0\}$.
\end{itemize}
A module $U$ is 1-cotilting if and only if $\Cogen U={}^\perp U$ (see \cite[Proposition~1.7]{ColDEsTon97}); in such a case $(\Ker\Hom(-,U),\Cogen U)$ is a torsion pair.

An injective cogenerator is a 1-cotilting bimodule. Therefore a Morita bimodule is a faithfully balanced 1-cotilting bimodule. The generalization is effective, for there exist rings with a faithfully balanced 1-cotilting bimodule which do not admit a Morita duality.

\begin{proposition}\label{prop:ereartcot}
Let $R$ be a left hereditary and left artinian ring. Then $_RR_R$ is a faithfully balanced $1$-cotilting bimodule.
\end{proposition}
\begin{proof} Since $R$ is perfect (\cite[Corollary~28.8]{AndFul74}), left hereditary implies right hereditary \cite[Corollary 2]{Sma67}, hence $\injdim (R_R)\leq 1$. Since $_RR$ is of finite length, then $R_R$ is product complete \cite[Theorem 4.1]{KraSao98}, i.e., $\Prod R_R=\Add R_R$, where $\Add R_R$ denotes the 
subcategory of the direct summands of direct sums of copies of $U$.
Thus $\Ext^i_R(\Prod R_R, R_R)=\Ext^i_R(\Add R_R, R_R)=0$ and for any injective cogenerator $W$ in $\rMod R$ there exists an exact sequence $0\to R_1\to R_0 \to W\to 0$ with $R_1, R_0\in \Add R_R=\Prod R_R$. This implies that $R_R$ is a  $1$-cotilting module and that it is $\Sigma$-pure injective \cite[Lemma~1.2.23]{GobTrl06}; thus $R$ is right coherent \cite[Corollary 5.4]{ColManTon10}. Since $R$ is right coherent and left perfect, $_RR$ is product complete by \cite[Proposition~3.9]{KraSao98}; applying again the above argument, we get that also $_RR$ is $1$-cotilting. 
\end{proof}

\begin{example}\label{ex:noMorita}
Consider the following example proposed by W.~Xue \cite[Remark~2.9]{Xue92}. Let $D$ be a division ring admitting a division subring $C$ such that $\dim (D_C)$ is finite but $\dim ({_CD})$ is not. As proved by $Xue$, the triangular matrix ring $R=\left(\begin{array}{cc} D & D \\ 0 & C \end{array}\right)$  is  hereditary and artinian (on both sides), and $R$ does not admit a Morita duality. By Proposition~\ref{prop:ereartcot} the regular bimodule $_RR_R$ is a faithful balanced $1$-cotilting bimodule.  
\end{example}

In the sequel let us fix a 1-cotilting bimodule $_RU_S$;
we denote by $\Delta_R$ and $\Delta_S$ (or simply by $\Delta$ when the ring is clear from the context) the contravariant functors $\Hom_R(-,U)$ and $\Hom_S(-,U)$, by $\Gamma_R$ and $\Gamma_S$ (or simply by $\Gamma$) the contravariant functors $\Ext^1_R(-,U)$ and $\Ext^1_S(-,U)$, and by $\R\Delta_R$ and $\R\Delta_S$ (or simply by $\R\Delta$)  the derived functors of $\Delta_R$ and $\Delta_S$ between the derived categories $\D(R)$ and $\D(S)$.
We often denote by $\Delta^2$ both the compositions $\Delta_S\circ\Delta_R$ and $\Delta_R\circ\Delta_S$, by $\Gamma^2$ both the compositions $\Gamma_S\circ\Gamma_R$ and $\Gamma_R\circ\Gamma_S$, and by $(\R\Delta)^2$ both the compositions $\R\Delta_S\circ\R\Delta_R$ and $\R\Delta_R\circ\R\Delta_S$.
Since $\Im\Delta\subseteq\Cogen U\subseteq  U^\perp$, it is $\R(\Delta^2)=\R\Delta\circ\R\Delta=:(\R\Delta)^2$ \cite[Proposition~5.4]{Har66}.

 The pair of functors $(\Delta_R,\Delta_S)$ is right adjoint with units the evaluation maps
\[\delta:\text{id}_{R\lMod}\to \Delta^2:=\Delta_S\Delta_R,\quad
\delta:\text{id}_{\rMod S}\to \Delta^2:=\Delta_R\Delta_S\]
defined setting for each module $M$
\[\xymatrix@R=1pt{
\delta_M:M\ar[r]&\Delta^2 M\\
m\ar@{|->}[r]&\widetilde m:\Delta M\ar[r]&U\\
&{}\phi\ar@{|->}[r]&\widetilde m(\phi):=\phi(m)
}\]
Along this paper we will define the notions of $\Delta$-reflexive, $\Delta$-torsionless, $\Delta$-torsion, $\Delta$-dense, $\Delta$-linearly compact modules; all these are known in the literature as $U$-reflexive, $U$-torsionless, $U$-torsion, $U$-dense, $U$-linearly compact modules. We make this choice since in the sequel other notions of reflexivity, torsionless, torsion, density, linear compactness will be associated to the bimodule $_RU_S$.

\begin{definition}
A module $M$ is 
\begin{itemize}
\item \emph{$\Delta$-reflexive} if $\delta_M$ is an isomorphism;
\item \emph{$\Delta$-torsionless} if $\delta_M$ is a monomorphism;
\item \emph{$\Delta$-torsion} if $\Delta M=0$. 
\end{itemize}
\end{definition}
The classes of $\Delta$-torsion and $\Delta$-torsionless modules coincide with the torsion and torsionless modules associated to the torsion pair
$(\Ker\Delta, \Cogen U=\Ker\Gamma)$.

Let $M$ be a module. Since $\Delta(\delta_{M})\circ\delta_{\Delta(M)}=1_{\Delta(M)}$, then
\begin{itemize}
\item $\Delta(M)$ is $\Delta$-torsionless;
\item  if $M$ is $\Delta$-reflexive, also $\Delta(M)$ is $\Delta$-reflexive.
\end{itemize}

For each module $M$ we denote by $\Rej_U M$ the intersection $\bigcap\{\Ker f:f\in\Delta M\}$.
The $\Rej_U M$ and $M/\Rej_U M$ are the torsion part and the torsionless part of the module $M$ with respect to the torsion pair $(\Ker\Delta, \Ker\Gamma)$.

Since $\Im\Delta\subseteq\Cogen U=\Ker\Gamma$, we have 
$\Gamma\Delta^\ell=0$ for each $\ell\geq 1$. The same is not true in general interchanging the role of the functors $\Delta$ and $\Gamma$ (see Example~\ref{ex:Gabri}).

By \cite[Lemma~13.6]{Kel96} also the total derived functors $\R\Delta_R$ and $\R\Delta_S$ form a right adjoint pair. Let us denote by $\eta$ both the units
\[\eta:\text{id}_{\D(R)}\to (\R\Delta)^2:=\R\Delta_S\R\Delta_R,\quad \eta:\text{id}_{\D(S)}\to (\R\Delta)^2:=\R\Delta_R\R\Delta_S\]
defined setting for each complex $X^\bullet$
\[\xymatrix{X^\bullet\ar[d]^{\eta_{X^\bullet}\qquad=}&.... \ar[r]&P_{-1}\ar[d]^{\delta_{P_{-1}}}\ar[r]& P_0\ar[d]^{\delta_{P_{0}}}\ar[r]& P_1\ar[d]^{\delta_{P_{1}}}\ar[r]& ...\\
(\R\Delta)^2(X^\bullet)&...\ar[r]&\Delta^2(P_{-1})\ar[r]&\Delta^2(P_{0})\ar[r]&\Delta^2(P_{1})\ar[r]&...}\]
where $pX^\bullet:=.... P_{-1}\to P_0\to P_1\to ...$ is a homotopically projective resolution of $X^\bullet$.
\begin{definition}
A complex $X^\bullet$ is 
\begin{itemize}
\item \emph{$\D$-reflexive} if $\eta_{X^\bullet}$ is an isomorphism, i.e. if it induces isomorphisms between the cohomologies; 
\item \emph{$\D$-torsionless} if $\eta_{X^\bullet}$ induces monomorphisms between the cohomologies.
\end{itemize}
A module $M$ is
\begin{itemize}
\item \emph{$\D$-reflexive} or \emph{$\D$-torsionless} if it is $\D$-reflexive or $\D$-torsionless considered as a complex concentrated in degree 0.
\end{itemize}
\end{definition}

Let $X^\bullet$ be a complex. Since $\R\Delta(\eta_{X^\bullet})\circ\eta_{\R\Delta(X^\bullet)}=1_{\R\Delta(X^\bullet)}$, then
\begin{itemize}
\item $\R\Delta(X^\bullet)$ is $\D$-torsionless;
\item  if $X^\bullet$ is $\D$-reflexive, also $\R\Delta(X^\bullet)$ is $\D$-reflexive.
\end{itemize}

\begin{theorem}[Corollary 3.6, Example~4.6 \cite{ManTon10}]\label{two:coomologie}
Let $_RU_S$ be a 1-cotilting bimodule. A complex is $\D$-reflexive if and only if its cohomologies are $\D$-reflexive.
\end{theorem}

Therefore for characterizing the $\D$-reflexive complexes is sufficient to characterize the $\D$-reflexive modules. In \cite[Theorem~3.10]{ManTon10} we have proved that the $\D$-reflexive modules form an exact subcategory of the whole module category.

We can associate with each short exact sequence of left $R$-modules  \[0\to A\to B\to C\to 0\] a triangle $A\to B\to C\to A[1]$ in $\D(R)$.
 Applying $\R\Delta^2$ to this triangle and considering the corresponding long exact sequence of cohomologies we get the following commutative diagram
\[
\xymatrix{&0\ar[d]^{\eta^{-1}_C}\ar[r]& A\ar[d]^{\eta^{0}_A}\ar[r]& B\ar[d]^{\eta^{0}_B}\ar[r]& C\ar[d]^{\eta^{0}_C}\ar[r]& 0\\
...\ar[r]&H^{-1}\R\Delta^2 C\ar[r]&H^0\R\Delta^2 A\ar[r]& H^0\R\Delta^2 B\ar[r]& H^0\R\Delta^2 C\ar[r]& 0
}
\]
where $\eta^i_X$ denote the module morphism
\[H^i(\eta_X): H^i X\to H^i \R\Delta^2 X\]
for each left $R$-module $X$.
\begin{proposition}\cite[Theorem~1.2, Proposition~1.3]{Ton00}\cite[Proposition~4.2]{ManTon10}
\label{prop:cottheorem}
If $_RU_S$ is a 1-cotilting bimodule, then for each module $M$
we have the following diagram with exact row
\[\xymatrix{&&M\ar[d]^{\eta^0_M}\\
0\ar[r]&\Gamma^2 M\ar[r]^-{\alpha_M}& H^0\R\Delta^2(M)\ar[r]^-{\beta_M}& \Delta^2 M\ar[r]& 0}\]
with $\beta_M\circ\eta^0_M$ equal to the evaluation map $\delta_M$. Moreover
\[H^i\R\Delta^2(M)=\begin{cases}
      0& \text{if }i\not=0,-1, \\
      \Delta\Gamma M& \text{if }i=-1.
\end{cases}\]
Then a module $M$ is $\D$-reflexive if and only if 
\begin{itemize}
\item $\Delta\Gamma M=0$ and $\eta^{0}_M:M\to H^0\R\Delta^2 (M)$ is an isomorphism;
\end{itemize}
 it is $\D$-torsionless if and only if 
 \begin{itemize}
 \item $\eta^{0}_M:M\to H^0\R\Delta^2 (M)$ is a monomorphism.
 \end{itemize}
\end{proposition}

A 1-cotilting bimodule has a cogenerator-type property in the derived category:
\begin{corollary}\label{cor:cogeneratore}
If $H^i\R\Delta^2(M)=0$ for $i=0, -1$ then $M=0$.
\end{corollary}
\begin{proof}
If $M\in\Ker H^0\R\Delta^2$ we have $\Delta^2 M=0=\Gamma^2 M$. Since $0=H^1\R\Delta^2 M=\Delta\Gamma M$ and $\Gamma\Gamma M=0$ imply $\Gamma M=0$, and $\Delta^2 M=0$ implies $\Delta M=0$, we conclude $M=0$ by Definition~\ref{definizione:1cotilting}.
\end{proof}

In the sequel, if $M$ is $\Delta$-torsionless (resp. $M$ is $\Delta$-torsion) we will identify the maps $\eta^0_M$ and $\beta_M\circ \eta^0_M$ (resp. $\alpha_M^{-1}\circ\eta^0_M$) which differ only for the isomorphism $\beta_M$ (resp. $\alpha_M^{-1}$).

\begin{corollary}\label{cor:H0Delta}
If $M$ is $\Delta$-torsionless then $H^0\R\Delta^2 (M)=\Delta^2 M$; in particular it is $\D$-reflexive if and only if it is $\Delta$-reflexive.
If $N$ is $\Delta$-torsion then
$H^0\R\Delta^2 (M)=\Gamma^2 M$.
\end{corollary}

We can check if a module is $\D$-torsionless or $\D$-reflexive considering separately its torsion and torsionless part with respect to the torsion pair $(\Ker\Delta,\Ker\Gamma)$:

\begin{proposition}\label{prop:iffrej}
Let $_RU_S$ be a 1-cotilting bimodule and $M$ be a module. Then for each $i\leq 0$ we have that $\eta^i_M$ is a monomorphism (resp. an isomorphism) if and only if $\eta^i_{\Rej_U M}$ and $\eta^i_{M/\Rej_U M}$ are monomorphisms (resp. isomorphisms).
In particular 
\begin{enumerate}
\item $M$ is $\D$-torsionless if and only if $\Rej_U M$ is $\D$-torsionless;
\item $M$ is $\D$-reflexive if and only if $\Rej_U M$ and $M/\Rej_U M$ are both $\D$-reflexive.
\end{enumerate}
\end{proposition}
\begin{proof}
We consider separately the cases $i\leq-2$, $i=-1$ and $i=0$.\\
$i\leq -2$: the maps $\eta^i_L=0$ are natural isomorphisms for each module $L$ (see Proposition~\ref{prop:cottheorem}).\\ $i=-1$: since $H^{-1}\R\Delta^2(M/\Rej_U M)=\Delta\Gamma (M/\Rej_U M)=0$, the map $\eta^{-1}_{M/\Rej_U M}$ is always an isomorphism. From the exact sequence
\[\xymatrix{0\ar[r]&H^{-1}\R\Delta^2\Rej_U M\ar[r]&H^{-1}\R\Delta^2 M\ar[r]&
0}
\]
we get immediately the thesis.\\
$i=0$: consider the commutative diagram
\[\xymatrix{0\ar[r]& \Rej_U M \ar[d]^{\eta^0_{\Rej_U M}}\ar[r]& M\ar[r]\ar[d]^{\eta^0_{M}}&M/\Rej_U M\ar@{^(->}[d]^{\eta^0_{M/\Rej_U M}=\delta_{M/\Rej_U M}}\ar[r]&0\\
0\ar[r]&
H^{0}\R\Delta^2\Rej_U M\ar[r]&H^{0}\Delta^2 M\ar[r]&
H^{0}\R\Delta^2 M/\Rej_U M\ar[r]&0
}
\]
If $\eta^0_M$ is a monomorphism, then $\eta^0_{\Rej_U M}$ is a monomorphism. Conversely, if $\eta^0_{\Rej_U M}$ is a monomorphism, also $\eta^0_M$ is a monomorphism by diagram chasing.
If $\eta^0_M$ is an isomorphism, then $\eta^0_{\Rej_U M}$ is a monomorphism and $\eta^0_{M/\Rej_U M}$ is an epimorphism and hence an isomorphism. Therefore also $\eta^0_{\Rej_U M}$ is an isomorphism. The converse is clearly true.
\end{proof}

We collect in the following result some useful property of $\Delta$-torsion modules.
\begin{lemma}\label{prop:key1}
Let $_RU_S$ be a $1$-cotilting bimodule. Consider a $\Delta$-torsion module $N$. Then
\begin{enumerate}
\item $\R\Delta N=\Gamma N[-1]$;
\item $H^0\R\Delta^3 (N)=\Delta\Gamma^2 N$, $H^i\R\Delta^3 (N)=0$ for each $i\geq 2$ and there is a short exact sequence
\[0\to \Gamma^3 N\to H^1\R\Delta^3 (N)\to \Delta^2\Gamma N\to 0;\]
\item $H^1\R\Delta(\eta_N)\circ\eta^0_{\Gamma N}=1_{\Gamma N}$ and hence $\Gamma N\leq^{\oplus}H^1\R\Delta^3 (N)$.
\end{enumerate}
If $\Delta\Gamma N=0$,  then
\begin{enumerate}[resume]
\item $\Gamma(\eta^0_{N})\circ \eta^0_{\Gamma N}=1_{\Gamma N}$
and hence $\Gamma N\leq^{\oplus}\Gamma^3 N$;
\item $\Ker(\eta^0_N)$ is $\Delta$-torsionless.
\end{enumerate}
If $\Delta\Gamma^2 N=0$ then
\begin{enumerate}[resume]
\item $\Coker\eta^0_N$ is $\Delta$-torsion.
\end{enumerate}
\end{lemma}
\begin{proof}
Let $...\stackrel{d_{-2}}\to P_{-1}\stackrel{d_{-1}}\to P_0\to (N\to )0$ be a projective resolution of $N$. Let us denote by $K_i$ the kernel of $d_i$. 
Applying $\Delta$ to the short exact sequences
\[0\to \Im d_{-1}\to P_0\to N\to 0,\ 
0\to K_{-1}\to P_{-1}\to\Im d_{-1}\to 0\text{ and}\]
\[0\to K_{-i}\to P_{-i}\to K_{-i+1}\to 0, \quad i\geq 2\]
we get
\[0=\Delta N\to \Delta P_0\to \Delta \Im d_{-1}\to \Gamma N\to 0,\]
\[0\to \Delta \Im d_{-1}\to \Delta P_{-1} \to \Delta K_{-1}\to 0, \quad \text{ and}\]
\[0\to \Delta K_{-i+1}\to \Delta P_{-i} \to \Delta K_{-i}\to 0\]
Applying $\Delta$ again we get
\[\xymatrix{0\ar[r]& \Delta\Gamma N\ar[r]& \Delta^2 \Im d_{-1}\ar[rr]\ar@{->>}[dr]&& \Delta^2 P_0\ar[r]& \Gamma ^2N\ar[r]& 0\\
&&&J\ar@{^(->}[ru]
}\]
\[0\to \Delta^2 K_{-1}\to \Delta^2 P_{-1} \to \Delta^2 \Im d_{-1}\to 0\text{ and}\]
\[0\to \Delta^2 K_{-i}\to \Delta^2 P_{-i} \to \Delta^2 K_{-i+1}\to 0\]
Applying a third time $\Delta$ we get
\[0\to \Delta\Gamma^2 N\to \Delta^3 P_0\to \Delta J\to  \Gamma ^3N\to 0,\]
\[0\to\Delta J\to\Delta^3\Im d_{-1}\to\Delta^2\Gamma N\to 0\]
\[0\to \Delta^3 \Im d_{-1}\to \Delta^3 P_{-1} \to \Delta^3 K_{-1}\to 0\text{ and}\]
\[0\to \Delta^3 K_{-i+1}\to \Delta^3 P_{-i} \to \Delta^3 K_{-i}\to 0\]
1. Since
\[\R\Delta N=\qquad 0\to \Delta P_0\to \Delta P_{-1}\to ...\]
we have easily $H^0 \R\Delta N=0$, $H^1 \R\Delta N=\Gamma N$, and 
$H^i \R\Delta N=0$ for each $i\not= 0,1$.\\
2. Since
\[\R\Delta^3 N=\qquad 0\to \Delta^3 P_0\to \Delta^3 P_{-1}\to ...\]
we have $H^0 \R\Delta^3 N=\Delta\Gamma^2 N$ and $H^i \R\Delta^3 N=0$ for each $i\not=0,1$. Finally we have the following commutative diagram
\[\xymatrix{\Delta^3 P_0\ar@{=}[d]\ar[r]&\Delta J\ar@{^(->}[d]\ar[r]&\Gamma^3 N\ar@{^(-->}[d]\ar[r]&0\\
\Delta^3 P_0\ar[r]&\Delta^3\Im d_{-1}=\Ker(\Delta^3P_{-1}\to\Delta^3P_{-2})\ar[d]\ar[r]&H^1\R\Delta^3 N\ar@{-->>}[dl]\ar[r]&0\\
&\Delta^2\Gamma N\ar[d]\\
&0
}
\]
from which we get the short exact sequence
\[0\to\Gamma^3 N\to H^1\R\Delta^3 N\to \Delta^2\Gamma N\to 0\]
3. We have $\R\Delta(\eta_N)\circ \eta_{\R\Delta N}=1_{\R\Delta N}$ and $\R\Delta N=\Gamma N[-1]$; in particular
\[1_{\Gamma N}=1_{H^1\R\Delta N}=H^11_{\R\Delta N}=
H^1\R\Delta(\eta_N)\circ H^1\eta_{\R\Delta N}=\]
\[=
H^1\R\Delta(\eta_N)\circ H^1\eta_{\Gamma N[-1]}=
H^1\R\Delta(\eta_N)\circ \eta^0_{\Gamma N}
\]
Therefore $\eta^0_{\Gamma N}$ is a monomorphism, $H^1\R\Delta(\eta_N)$ is an epimorphism and $\Gamma N$ is a direct summand of $H^1\R\Delta^3 N$.\\
4. Since $\Gamma N$ is $\Delta$-torsion, by (2) it is $H^1\R\Delta^3 N=\Gamma^3 N$ and we get easily that $H^1\R\Delta(\eta_N)=\Gamma(\eta^0_N)$. Then we conclude by (3).\\
5. Consider the exact sequence
\[\xymatrix{
(*)&0\ar[r]& \Ker\eta^0_N\ar[r]& N\ar[rr]^{\eta^0_N}\ar@{->>}[rd]&& \Gamma^2 N\ar[r]& \Coker \eta^0_N\ar[r]& 0\\
&&&&I\ar@{^(->}[ru]
}
\]
Applying the functor $\Delta$ we get the exact sequences
\[\Gamma^3 N\to \Gamma I\to 0\quad\text{and}\quad
\Gamma I\to \Gamma N\to \Gamma \Ker\eta^0_N\to 0\]
Since $\Gamma(\eta^0_N)$ is an epimorphism by (4), we conclude $\Gamma \Ker\eta^0_N=0$.\\
6. Applying $\Delta$ to $(*)$ we get $0\to\Delta\Coker\eta^0_N\to \Delta\Gamma^2 N=0$.
\end{proof}

\section{$\D$-torsionless modules}\label{sec:Driflmodules}

In this section we study the $\D$-torsionless modules. In particular we will see that they are in general different from the whole category of modules.

In the sequel we assume always $_RU_S$ is a 1-cotilting bimodule (see Definition~\ref{definizione:1cotilting}). 
Let us denote by $\D \Cogen U$ the class of $\D$-torsionless modules, i.e. of modules $M$ such that $\eta^0_M$ is a monomorphism (see Proposition~\ref{prop:cottheorem}). Clearly $\Cogen U\subseteq \D \Cogen U$. The class of $\D$-torsionless modules is a pretorsion-free class, i.e. it is closed under submodules and products; in general it is not closed under extensions.

\begin{lemma}\label{chiusoestpart}
Given an exact sequence $\xymatrix{0\ar[r]&L\ar[r]&M\ar[r]&N\ar[r]&0}$, if $L$ and $N$ belong to $\D \Cogen U$ and $\Delta\Gamma N=0$, then also $M$ belongs to $\D \Cogen U$.
\end{lemma}
\begin{proof}
Consider the commutative diagram with exact rows
\[\xymatrix{0\ar[r]&L\ar[d]^{\eta^0_L}\ar[r]&M\ar[d]^{\eta^0_M}\ar[r]&N\ar[d]^{\eta^0_N}\ar[r]&0\\
\Delta\Gamma N=0\ar[r]&H^0\R\Delta^2 L\ar[r]&H^0\R\Delta^2 M\ar[r]&H^0\R\Delta^2 N\ar[r]&0}\]
If $L$ and $N$ are $\D$-torsionless, also $\eta^0_M$ is a monomorphism. 
\end{proof}

Remembering that $\D\Cogen U$ is closed under submodules and products, the following result give a large number of $\D$-torsionless modules.
\begin{proposition}\label{prop:Dtorsionless}
For each module $M$, the modules $\Delta M$, $\Gamma M$ and $H^0\R\Delta^2 M$ are $\D$-torsionless. If $\Delta\Gamma^3 M=0=\Delta\Gamma M$ then also
$M$ is $\D$-torsionless.
\end{proposition}
\begin{proof}
Clearly $\Delta M\in\Cogen U\subseteq\D\Cogen U$. By Lemma~\ref{prop:key1}, (3), $\eta^0_{\Gamma N}$ is a monomorphism. From the exact sequence
\[\xymatrix{0\ar[r]&\Gamma^2 M\ar[r]&H^0\R\Delta^2 M\ar[r]&\Delta^2 M\ar[r]&0}\]
and Lemma~\ref{chiusoestpart}, we obtain that also $H^0\R\Delta^2 M$ is $\D$-torsionless. Let $\Delta\Gamma^3 M=0=\Delta\Gamma M$; to prove that $M$ is $\D$-torsionless it is not restrictive to assume $M$ is $\Delta$-torsion (see Proposition~\ref{prop:iffrej}).
By Lemma~\ref{prop:key1}, (5), $\Gamma\Ker\eta^0_M=0$; then applying $\Delta$ to the diagram 
\[\xymatrix{
0\ar[r]& \Ker\eta^0_M\ar[r]& M\ar[rr]^{\eta^0_M}\ar@{->>}[rd]&& \Gamma^2 M\ar[r]& \Coker \eta^0_M\ar[r]& 0\\
&&&I\ar@{^(->}[ru]
}
\]
we get the exact sequences
\[\Delta M=0\to\Delta \Ker\eta^0_M\to\Gamma I\to \Gamma M \to\Gamma \Ker\eta^0_M=0\quad\text{and}\]
\[\Delta I=0\to \Gamma \Coker \eta^0_M\to\Gamma^3 M\to \Gamma I\to 0.\]
Applying again $\Delta$ we get
\[\Delta\Gamma M=0\to\Delta\Gamma I\to\Delta^2 \Ker\eta^0_M\to \Gamma^2 M\to \Gamma^2 I\to 0\quad\text{and}\]
\[0\to\Delta\Gamma I\to \Delta\Gamma^3 M=0\to \Delta \Gamma \Coker \eta^0_M\to \Gamma^2 I\to \Gamma^4 M\to \Gamma^2 \Coker \eta^0_M\to 0\]
Since $\Gamma(\eta^0_{\Gamma M})\circ \Gamma^2(\eta^0_M)=\Gamma(\Gamma(\eta^0_M)\circ\eta^0_{\Gamma M})=
\Gamma(1_{\Gamma M})=1_{\Gamma^2 M}$, 
$\Gamma^2(\eta^0_M)$ is a monomorphism. The map $\Gamma^2(\eta^0_M)$ is the composition
of the epimorphism $\Gamma^2 M\to\Gamma^2 I$ and of $\Gamma^2 I\to\Gamma^4 M$.
Since $\Delta\Gamma I=0$, we get $\Delta^2 \Ker\eta^0_M=0$; then also $\Delta \Ker\eta^0_M=0$ and hence $\Ker\eta^0_M=0$ by Definition~\ref{definizione:1cotilting}.
\end{proof}

Any given class $\mathcal C$ of objects cogenerates a torsion pair \cite[Chapter~VI, \S 2]{Ste75} whose torsion free class is the smallest torsion free class containing $\mathcal C$.

\begin{proposition}
Let $_RU_S$ be a 1-cotilting bimodule. The class $\Ker H^0\R\Delta^2$ is the torsion class cogenerated by $\D\Cogen U$; it coincides with the class of modules $N$ such that $\eta^0_N=0$.
\end{proposition}
\begin{proof}
Since $\Im H^0\R\Delta^2\subseteq\D\Cogen U$ we have
\[\{N:H^0\R\Delta^2N=0\}\subseteq \{N:\eta^0_N=0\}\subseteq \{N:\Hom(N, M)=0\  \forall M \in \D\Cogen U\}\]
Assume  $\Hom(N, M)=0$ for each $M \in \D\Cogen U$.
First of all $N$ is $\Delta$-torsion: if $\Delta N\not=0$ there exists a non zero map between $N$ and $U\in\Cogen U\subseteq \D\Cogen U$. Then $H^0\R\Delta^2N=\Gamma^2 N$; since $\Gamma^2 N$ belongs to $\D\Cogen U$ by Proposition~\ref{prop:Dtorsionless}, we have $\eta^0_N=0$. Let us prove that $\Gamma N$ belongs to $\Ker\Gamma$ and hence $H^0\R\Delta^2 N=\Gamma^2 N=0$. Applying $\Delta$ to the diagram
\[\xymatrix{&0\ar[r]&K\ar[d]\ar[r]&P\ar[d]\ar[r]&N\ar[r]\ar[d]^{\eta^0_N=0}&0\\
0\ar[r]&\Delta\Gamma N\ar[r]&\Delta^2 K\ar@{->>}[dr]\ar[r]&\Delta^2 P\ar[r]&\Gamma^2 N\ar[r]&0\\
&&&I\ar@{^(->}[u]}\]
we get the commutative diagram
\[\xymatrix{\Delta P\ar[rr]&&\Delta K\ar[rr]&&\Gamma N\ar[r]&0\\
\Delta^3 P\ar[rr]\ar[u]^{\Delta(\delta_P)}&&\Delta^3 K\ar[rr]\ar[u]^{\Delta(\delta_K)}&&H^1\R\Delta^3(N)\ar[r]\ar[u]_{H^1\R\Delta(\eta_N)}&0\\
\Delta^3 P\ar[r]\ar@{=}[u]&\Delta I\ar[ru]\ar[rd]\ar[rr]|!{[dr];[ur]}\hole&&\Gamma^3 N\ar@{^(->}[rd]\ar@/^/[uur]^>>>>>{\Gamma(\eta^0_N)=0}|!{[ul];[ur]}\hole\ar[rr]|!{[dr];[ur]}\hole&&0\\
\Delta^3 P\ar[rr]\ar@{=}[u]&&\Delta^3 K\ar[rr]\ar@{=}[uu]&&H^0\R\Delta^2(\Gamma N)\ar[r]\ar@{->>}[rd]^{\beta_{\Gamma N}}\ar@{=}[uu]&0\\
&&&&&\Delta^2\Gamma N\\
\Delta P\ar[rr]\ar[uu]^{\delta_{\Delta P}}&&\Delta K\ar[rr]\ar[uu]^{\delta_{\Delta K}}&&\Gamma N\ar[r]\ar[uu]^{\eta^0_{\Gamma N}}&0}
\]
In particular $\Gamma^3 N$ is contained in $\Ker H^1\R\Delta(\eta_N)$.
Since $H^1\R\Delta(\eta_N)\circ \eta^0_{\Gamma N}=1_{\Gamma N}$ by Lemma~\ref{prop:key1}, (3), we have $\Im \eta^0_{\Gamma N}\cap \Gamma^3 N=0$ and hence $\beta_{\Gamma N}\circ\eta^0_{\Gamma N}$ is a monomorphism. Therefore $\Gamma N$ is $\Delta$-torsionless, i.e., it belongs to $\Ker\Gamma$.
\end{proof}

The torsion class $\Ker H^0\R\Delta^2$ contains all modules $M$  such that $\Delta^2 M=0=\Gamma^2 M$. In general $\Ker H^0\R\Delta^2\not =0$ and hence also the torsion free class cogenerated by $\D\Cogen U$ is a  proper subcategory of the whole category of modules:

\begin{example}[\cite{DEs02} Theorem~2.5]\label{ex:Gabri}
Let $A$ be the generalized Kronecker algebra of dimension $d$ over an algebraically closed field $K$. $A$ is isomorphic to $\left(\begin{array}{cc}K & 0 \\V & K\end{array}\right)$ where $V$ is a vector space of dimension an infinite cardinal $d$. By \cite[Lemma~2.2]{DEs02}, $_AA_A$ is a faithfully balanced cotilting bimodule. Following \cite[Theorem~2.5]{DEs02}, for every cardinal $c$ such that $1\leq c\leq d$, there is an indecomposable cyclic $A$-module $Y$ such that $\dim_K(Y)=c$, $\Delta Y=0$ and $\Gamma Y$ is a free module of uncountable rank: in particular $Y$ is $
\Delta$-torsion, $H^0\R\Delta^2 Y=\Gamma^2 Y=0$ and $\Delta\Gamma Y\not=0$.
\end{example}



\section{$\D$-reflexive modules}

Given a module $L$, we denote by $\gen L$ the class of all modules $M$ such that there exist a natural number $n\in\mathbb N$ and an epimorphism $\xymatrix{\phi^L_M:L^n\ar@{->>}[r]&M}$. We denote by $\pres L$ the class of modules in $\gen L$ such that $\Ker\phi^L_M$ belongs to $\gen L$.

\begin {lemma}\label{lemma:fingengenerale}
Let $_RU_S$ be a $1$-cotilting bimodule with $S=\End({}_RU)$. 
\begin{enumerate}
\item If $F$ belongs to $\gen S_S$ (resp. $\gen {}_RU$), then the morphism $\eta^0_F$ is epic; if $F$ is also $\D$-torsionless, then $\eta^0_F$ is an isomorphism. In particular
\begin{enumerate}
\item if $F$ is $\Delta$-torsionless, then $F$ is $\D$-reflexive.
\item if $F$ is a submodule of $\Gamma X$ for some $X$ in $R\lMod$ (resp. in $\rMod S$) then $\eta^0_F$ is an isomorphism; if further $\Delta\Gamma^2 X=0$ then $F$ is $\D$-reflexive.
\item if $F$ belongs to $\pres S_S$  (resp. $\pres {}_RU$), then $F$ is $\D$-reflexive.
\end{enumerate}
\item A simple right $S$-module $L$ is $\D$-reflexive if and only if $\Delta\Gamma L=0$.
\end{enumerate}
\end{lemma}
\begin{proof}
1) Let $F\in \gen S_S$; then there exists a short exact sequence
\[0\to K\to S^n\to F\to 0\]
for a suitable $n\in\mathbb N$. To this exact sequence we can associate the triangle in $\D(S)$
\[K\to S^n\to F\to K[1]\]
Applying the functor $\R\Delta^2$ and considering the long exact sequences of cohomologies, we get the following commutative diagram
\[\xymatrix{
&0\ar[r]&K\ar[d]^{\eta^0_K}\ar[r]&S^n\ar[d]^\cong\ar[r]& F\ar[d]^{\eta^0_F}\ar[r]&0\\
0\ar[r]&\Delta\Gamma F\ar[r]&
H^0\R\Delta^2 K\ar[r]&\Delta^2 S^n=S^n\ar[r]&H^0\R\Delta^2 F\ar[r]&0}
\]
getting immediately the surjectivity of $\eta^0_F$.\\
1a) If $F$ is $\Delta$-torsionless, then $F$ is $\D$-torsionless and $\Delta\Gamma F=0$. Therefore $F$ is $\D$-reflexive by Proposition~\ref{prop:cottheorem}.\\
1b) By Proposition~\ref{prop:Dtorsionless} and the closure of $\D\Cogen U$ with respect to submodules, $\eta^0_F$ is an isomorphism. Applying $\Delta$ to the exact sequence $\Gamma^2 X\to \Gamma F\to 0$ we get $0\to\Delta\Gamma F\to \Delta\Gamma^2 X$. If the latter is zero, then $F$ is $\D$-reflexive.
\\
1c) If $K$ belongs to $\gen S_S$ or to $\gen {}_RU$, then $\eta^0_K$ is an epimorphism; it is also a monomorphism and hence an isomorphism. Therefore $\Delta\Gamma F=0$ and $\eta^0_F$ is an isomorphism.\\
The cases $F\in\gen {}_RU$ and $F\in\pres {}_RU$ are analogous.\\
2) The necessity of the condition follows by the definition of $\D$-reflexive module. Conversely assume $\Delta\Gamma L=0$. If $L=0$ it is trivially $\D$-reflexive. Assume $L\not=0$; since $L$ is finitely generated, then $\eta^0_L$ is an epimorphism by 1. If $\eta^0_L$ is not a monomorphism, then $\eta^0_L=0$ and hence $H^0\R\Delta^2 L=0$. By Corollary~\ref{cor:cogeneratore} we get $L=0$: contradiction.
\end{proof}

The following result establishes an important step in the characterization of $\D$-reflexive modules in terms of a linear compactness notion.
\begin{proposition}\label{prop:lincomp}
Let $_RU_S$ be a 1-cotilting bimodule and $(M\stackrel{p_{\lambda}}{\to} N_{\lambda})_{\lambda\in \Lambda}$ be an inverse system of epimorphisms. If  all the maps $\eta^0_{M}$ and $\eta^0_{N_{\lambda}}$ are isomorphisms, then $ M\to \varprojlim N_{\lambda}$ is an epimorphism.
\end{proposition}
\begin{proof}
Let us consider the short exact sequence
\[(*)\qquad 0\to K_\lambda:=\Ker p_{\lambda}\to M\stackrel{p_{\lambda}}{\to} N_{\lambda}\to 0.\]
First let us prove the statement for $M$ being $\Delta$-torsion. In such a case also the modules $N_{\lambda}$ are $\Delta$-torsion and the map $\Gamma^2(M)\stackrel{\Gamma^2(p_{\lambda})}{\to} \Gamma^2(N_{\lambda})$ is an epimorphism. Applying $\Delta$ to $(*)$ we get the exact sequences
\[0\to \Delta(K_{\lambda})\to \Gamma(N_{\lambda})\to I_{\lambda}\to 0\text{ and }0\to I_{\lambda}\to \Gamma M \to \Gamma(K_{\lambda})\to 0.\]
Since $U$ is pure injective \cite{Baz03}, $\Cogen U$ is closed under direct limits \cite[Proposition~1.3]{ManRuzTon03}. Applying first the direct limit and then $\Delta$, we  obtain the epimorphism \[\Gamma^2(M)\to \Gamma(\varinjlim \Gamma(N_{\lambda}))=\varprojlim\Gamma^2 N_\lambda\to 0.\]
From the commutative diagram with exact rows
\[
\xymatrix{
{} \Gamma^2(M)\ar[r]\ar@{=}[d]&\Gamma(\varinjlim\Gamma N_\lambda)\ar[r]&0\\
{}\Gamma^2 M\ar[r]&\varprojlim \Gamma^2 N_\lambda\ar[r]\ar@{=}[u]& 0\\
 M\ar[u]^\cong_{\eta^0_M}\ar[r]^{\varprojlim p_\lambda}&\varprojlim N_\lambda\ar[u]^\cong_{\varprojlim \eta^0_{N_\lambda}}
}\]
we conclude that ${\varprojlim p_\lambda}$ is an epimorphism.

\noindent Assume now $M$ is $\Delta$-torsionless. From the exact sequence $0\to K_{\lambda}\to M\to N_{\lambda}\to 0$ we get the long exact sequence
\[\Delta\Gamma N_\lambda=H^1\R\Delta^2 N_\lambda\to \Delta^2 K_\lambda\to\Delta^2 M\to H^0\R\Delta^2(N_{\lambda})\to 0\]
Denoted by $I_{\lambda}$ the cokernel of $\Delta p_\lambda$, we have the following commutative diagram
\[
\xymatrix{0 \ar[r]& \Delta(I_{\lambda})\ar[d]\ar[r] & \Delta^2(M)\ar[d] \ar[r] & \Delta^2(N_{\lambda})\ar@{=}[d] \ar[r] & 0\\
0 \ar[r] & \Gamma^2(N_{\lambda})\ar[r] \ar[d]& H^0\R\Delta^2(N_{\lambda})\ar[r] \ar[d]& \Delta^2(N_{\lambda})\ar[r] & 0\\
& 0 & 0}
\] 
Consider the exact sequences 
$$0\to \varinjlim \Delta(N_{\lambda})\to \Delta(M)\to \varinjlim I_{\lambda}\to 0 \text{ and}$$
$$0\to \varinjlim I_{\lambda}\to \varinjlim \Delta(K_{\lambda})\to \varinjlim \Gamma(N_{\lambda})\to 0.$$
Since $I_\lambda\in\Ker\Gamma_S$ and the latter is closed under direct limits, we get the exact sequence
$$0\to \Delta(\varinjlim I_{\lambda})\cong \varprojlim \Delta(I_{\lambda})\to \Delta^2(M)\to \Delta(\varinjlim \Delta(N_{\lambda}))\cong \varprojlim \Delta^2(N_{\lambda})\to 0$$ and the epimorphism
$$\Delta(\varinjlim I_{\lambda})\cong \varprojlim \Delta(I_{\lambda})\to \Gamma(\varinjlim \Gamma(N_{\lambda}))\cong \varprojlim \Gamma^2(N_{\lambda})\to 0.$$
Passing to inverse limits in the previous diagram, we obtain the commutative diagram with exact rows
$$
\xymatrix{0 \ar[r]& \varprojlim \Delta(I_{\lambda})\ar[d]\ar[r] & \Delta^2(M)\ar[d] \ar[r] & \varprojlim\Delta^2(N_{\lambda})\ar@{=}[d] \ar[r] & 0\\
0 \ar[r] & \varprojlim \Gamma^2(N_{\lambda})\ar[r] \ar[d]& \varprojlim H^0\R\Delta^2(N_{\lambda})\ar[r] &\varprojlim  \Delta^2(N_{\lambda})\\
& 0 & }
$$
which implies that $\varprojlim H^0\R\Delta^2(p_{\lambda})$ is an epimorphism. Hence from 
$$
\xymatrix{M\ar[d]^\cong_{\eta^0_M}\ar[rrrr] &&&& \varprojlim N_{\lambda} \ar[d]^\cong_{\varprojlim \eta^0_{N_{\lambda}}} & \\
\Delta^2(M)\ar[rrrr]^{\varprojlim H^0\R\Delta^2(p_{\lambda})} &&& &\varprojlim H^0\R\Delta^2(N_{\lambda}) \ar[r] & 0}
$$
we conclude that $\varprojlim p_{\lambda}$ is an epimorphism.

\noindent Finally, we consider the general case. Let $M$ be a module; consider its $\Delta$-torsion and $\Delta$-torsionless parts and the commutative diagram
$$
\xymatrix{0 \ar[r] &  \Rej_U M \ar[r]\ar[d]^{t_{\lambda}} &  M \ar[r]\ar[d]^{p_{\lambda}} &  M/{\Rej_U M}\ar[r]\ar[d]^{s_{\lambda}} & 0\\
0 \ar[r] &  \Rej_U N_{\lambda} \ar[r] &  N_{\lambda} \ar[r]\ar[d] &  N_{\lambda}/{\Rej_U N_{\lambda}}\ar[r]\ar[d] & 0\\
& & 0 & 0}
$$
It easy to check that the $\eta^0$'s of all the modules involved in the above diagram are isomorphisms. 
In order to get the thesis, we have to show that $\varprojlim^{(1)} \ker p_{\lambda}=0$: indeed, $\varprojlim^{(1)}M=0$ since trivially $(M)_{\lambda\in\Lambda}$ is a \emph{weakly flabby} inverse system \cite[Theorem~1.8]{Jen72}. Applying to the above diagram the snake lemma we get the following diagram
\[\xymatrix@-1pc{
0\ar[r] &\Ker t_\lambda\ar[r]& \Ker p_\lambda\ar[rr]\ar@{->>}[dr]&&
\Ker s_\lambda\ar[r]\ar@{^(->}[d]&\Coker t_\lambda\ar[r]& 0\\
&&0\ar[r]&J_\lambda\ar@{^(->}[ur]\ar[r]_-{j_\lambda}&M/\Rej_U M\ar@{->>}[d]\ar[r]_-{q_\lambda}&\Coker j_\lambda=: L_\lambda\ar[r]&0\\
&&&&N_\lambda/\Rej_UN_\lambda}\]
Then we get the exact sequence
$$0\to {\Coker t_{\lambda}}\to L_{\lambda}\to N_{\lambda}/{\Rej_U N_{\lambda}}\to 0$$
The maps $\eta^0_{L_{\lambda}}$ and $\eta^0_{\Im t_{\lambda}}$ are isomorphisms. Indeed, consider the exact sequence
\[0\to \Im t_\lambda\to \Rej_U N_{\lambda}\to \Coker t_\lambda\to 0;\]
 $\Im t_{\lambda}$ is a submodule of $\Rej_U N_{\lambda}$ and a factor of $\Rej_U M$; since $\eta^0_{\Rej_U N_{\lambda}}$ and $\eta^0_{\Rej_U M}$ are isomorphisms, and hence respectively a monomorphism and an epimorphism, also $\eta^0_{\Im t_{\lambda}}$ is an isomorphism. It follows that  $\eta^0_{\Coker t_{\lambda}}$ is an isomorphism. Since $\eta^0_{N_{\lambda}/{\Rej_U N_{\lambda}}}$ is an isomorphism, we get that also $\eta^0_{L_{\lambda}}$  is an isomorphism. Thus we can apply the results  proved above to get that $\varprojlim t_{\lambda}$ is an epimorphism and hence $\varprojlim^{(1)} \ker t_{\lambda}=0$, and $\varprojlim q_{\lambda}$ is an epimorphism and hence $\varprojlim^{(1)} J_{\lambda}=0$. So we conclude  that $\varprojlim^{(1)} \ker p_{\lambda}=0$.
\end{proof}

Thanks to the previous proposition we can generalize to the 1-cotilting case a wellknown result of Osofsky for Morita dualities \cite[Lemma~13]{Oso1966}.
\begin{proposition}\label{prop:Osofsky}
Let $_RU_S$ be a 1-cotilting bimodule and $M$ an infinite direct sum of non-zero modules. Then $\eta^0_M$ is not an isomorphism; in particular $M$ can not be $\D$-reflexive.
\end{proposition}
\begin{proof}
Let $M=\oplus_{i\in I}M_i$ be a direct sum of non-zero modules. For each finite subset $F\subseteq I$, set $M_{(F)}=\oplus_{j\not\in F}M_j$ and $M_F=\oplus_{j\in F}M_j$. Consider the inverse system of epimorphisms $(M\stackrel{p_F}{\to}M_F)_F$, where the $p_F$'s are the canonical projections, and the following commutative diagrams
\[\xymatrix@-1pc{0\ar[r]&M_F\ar[d]^{\eta^0_{M_F}}\ar[r]& M\ar[d]^{\eta^0_{M}}\ar[r]& M_{(F)}\ar[d]^{\eta^0_{M_{(F)}}}\ar[r]&0\\
&H^{0}\R\Delta^2M_F\ar[r]&H^{0}\R\Delta^2M\ar[r]&H^{0}\R\Delta^2M_{(F)}\ar[r]&0}
\]
\[\xymatrix@-1pc{0\ar[r]&M_{(F)}\ar[d]^{\eta^0_{M_{(F)}}}\ar[r]& M\ar[d]^{\eta^0_{M}}\ar[r]& M_F\ar[d]^{\eta^0_{M_{F}}}\ar[r]&0\\
&H^{0}\R\Delta^2M_{(F)}\ar[r]&H^{0}\R\Delta^2M\ar[r]&H^{0}\R\Delta^2M_{F}\ar[r]&0}
\]
Assume $\eta^0_M$ is an isomorphism; then both $\eta^0_{M_{(F)}}$ and $\eta^0_{M_{F}}$ are isomorphisms. Then, by Proposition~\ref{prop:lincomp}, 
the inverse limit
\[\varprojlim p_F:M\to\varprojlim M_F\]
is surjective. Let $0\not=m_i\in M_i$ for each $i\in I$, and $m_F\in M_F$ the element whose $j$-component is $m_j$ for each $j\in F$.
Clearly $(m_F)_{F\subseteq I}$ is an element of $\varprojlim M_F$; therefore there exists $m\in M$ such that $p_F(m)=m_F$ for each finite subset $F\subseteq I$. This element $m$ has all its components different from zero: therefore $I$ is finite.
\end{proof}

\begin{corollary}
Let $_RU_S$ be a 1-cotilting bimodule with $S=\End({}_RU)$. A semisimple right $S$-module $M$ is $\D$-reflexive if and only if it is finitely generated and $\Delta\Gamma M=0$.
\end{corollary}
\begin{proof}
By Propositions~\ref{prop:cottheorem}, \ref{prop:Osofsky} the conditions $M$ finitely generated and $\Delta\Gamma M=0$ are clearly necessary. They are also sufficient. Let $M=\bigoplus_{i=1}^n L_i$ with $L_i$ simple right $S$-modules; then $0=\Delta\Gamma M=\bigoplus_{i=1}^n \Delta\Gamma L_i$ implies by Lemma~\ref{lemma:fingengenerale} that the $L_i$'s are $\D$-reflexive and hence $M$ is $\D$-reflexive.
\end{proof}

If a ring involved in a cotilting duality is noetherian, then, as in the Morita setting, it is necessarily semiperfect:

\begin{proposition}\label{prop:semiperfect}
Let $_RU_S$ be a 1-cotilting bimodule with $S=\End ({}_RU)$. If $S$ is noetherian, then it is linearly compact and hence semiperfect.
\end{proposition}
\begin{proof}
Let $\xymatrix{S\ar@{->>}[r]^{p_\lambda}&N_\lambda}$, $\lambda\in\Lambda$, be an inverse system of epimorphisms. By Lemma \ref{lemma:fingengenerale}, (1c), $\eta^0_S$ and $\eta^0_{N_\lambda}$ are isomorphisms; then by Proposition~\ref{prop:lincomp} also $\varprojlim p_\lambda:S\to \varprojlim N_\lambda$ is an epimorphism and hence $S_S$ is linearly compact. Then by
\cite[Corollary~3.14]{Xue92} we conclude $S$ is semiperfect.
\end{proof}

\begin{proposition}\label{prop:arti}
Let $_RU_S$ be a $1$-cotilting bimodule with $S=\End ({}_RU)$. If $R$ is a left artinian ring, then the $\D$-reflexive left $R$-modules are finitely generated. In particular $_RU$ itself is finitely generated.
\end{proposition}
\begin{proof}
Given a $\D$-reflexive left $R$-module $_RM$, let us consider the semisimple module $M/JM$. Clearly $\eta^0_{M/JM}$ is an epimorphism; it is also a monomorphism. Indeed, assume $\Ker\eta^0_{M/JM}\not=0$ and let $L\not=0$ be a simple submodule of $\Ker\eta^0_{M/JM}$; by Lemma~\ref{lemma:fingengenerale}, (1c), $L$ is $\D$-reflexive and hence $\eta^0_L$ is an isomorphism. Since the inclusion $i:L\to M/JM$ is a splitting monomorphism, also $H^0\R\Delta^2(i):H^0\R\Delta^2 L\to H^0\R\Delta^2 (M/JM)$ is a splitting mono. Then by
\[H^0\R\Delta^2(i)\circ\eta^0_L=\eta^0_{M/JM}\circ i=0\]
we get $L=0$: contradiction. Thus $\eta^0_{M/JM}$ is an isomorphism and, by Proposition~\ref{prop:Osofsky}, $M/JM$ is finitely generated. Since an artinian ring is semiprimary, $JM$ is superfluous in $M$ and hence $M$ is finitely generated by \cite[Theorem~10.4]{AndFul74}.
\end{proof}
\begin{corollary}
Let $_RU_S$ be a faithfully balanced $1$-cotilting bimodule. If $R$ is a left artinian ring, then the $\D$-reflexive left $R$-modules are exactly the finitely generated modules.
\end{corollary}
\begin{proof}
Since $R=\End (U_S)$, the finitely generated left $R$-modules are $\D$-reflexive. Then we conclude by Proposition~\ref{prop:arti}.
\end{proof}
\section{Characterization of $\D$-reflexive modules}
In this section we characterize the $\D$-reflexive modules in terms of a linear compactness and a density notions. Proposition~\ref{prop:iffrej} suggests us to characterize separately the $\Delta$-torsion and the $\Delta$-torsionless $\D$-reflexive modules.

\subsection{$\D$-reflexive $\Delta$-torsionless modules}\label{sec:Drifltorsionless}

If $M$ is a $\Delta$-torsionless module, then it is $\D$-torsionless and $\Delta\Gamma M=0$: in particular by Lemma~\ref{lemma:fingengenerale} all finitely generated $\Delta$-torsionless modules are $\D$-reflexive.
\begin{lemma}
If $_RU_S$ is a 1-cotilting bimodule and $M$ is $\Delta$-torsionless, then $\Coker\eta^0_M$ belongs to $\Ker\Gamma$.
\end{lemma}
\begin{proof}
By Proposition~\ref{prop:cottheorem} we have $\Coker\eta^0_M\cong\Coker \delta_M$. Applying $\Delta$ to the short exact sequence
\[0\to M\stackrel{\delta_M}{\to} \Delta^2 M\to \Coker \delta_M\to 0\]
we get $\Delta^3 M\stackrel{\Delta(\delta_M)}{\to}\Delta M\to \Gamma \Coker \delta_M\to 0$. Since $\Delta(\delta_M)$ is an epimorphism, we conclude $\Gamma \Coker \delta_M=0$.
\end{proof}

\begin{definition}\label{defi:denseDeltalc}
A module $M$ is 
\begin{itemize}
\item \textbf{$\Delta$-dense} in $\Delta^2 M$ ($\Delta$-dense for short) if for any $h\in\Delta^2 M$ and any finitely generated submodule $L$ of $\Delta M$, there exists $m_L\in M$ such that $h(\ell)=\delta_M(m_L)(\ell)$ for each $\ell\in L$ \cite[\S 47]{Wis91};
\item \textbf{$\Delta$-linearly compact}
if $M\in\Cogen U$ and $\varprojlim p_\lambda: M\to \varprojlim L_\lambda$ is an epimorphism for every inverse system of epimorphisms $p_\lambda:M\to L_\lambda$ with the $L_\lambda$'s in $\Cogen U$.
\end{itemize}
\end{definition}
\begin{remark}\label{rem:Deltadense}
A module $M$ is $\Delta$-dense if and only if for any finitely generated submodule $L$ of $\Delta M$, denoted by $\xymatrix{L\ar@{^(->}[r]^i& \Delta M}$ the inclusion, we have $\Im(\Delta i\circ\delta^0_M)=\Im\Delta(i)$.
\end{remark}
The main result about $\Delta$-torsionless modules is the following characterization, consequence of results in \cite{Baz03} and \cite{ManRuzTon03}:
\begin{theorem}[\cite{ColFul0610} Proposition~5.3.7]\label{teo:deltariflessivo}
If $_RU_S$ is a 1-cotilting bimodule, then a $\Delta$-torsionless module $M$ is $\D$-reflexive if and only if it is \emph{$\Delta$-dense} and \emph{$\Delta$-linearly compact}. 
\end{theorem}

If $_RU$ and $U_S$ are cogenerators in the respective categories, all modules are clearly $\Delta$-torsionless and $\Delta$-dense by \cite[\S 47.6]{Wis91}, and the $\Delta$-linear compactness coincides with the usual linear compactness: thus we deduce the classical M\"uller characterization of \emph{Morita-reflexive objects} in the case $_RU_S$ is a Morita bimodule \cite{Mul70}.

Comparing with the classical Morita case, the $\Delta$-density condition appears. In \cite[47.7, (1)]{Wis91} is proved that if  for any $k\in\mathbb N$ and $f:M\to U^k$, the $\Coker f$ is $\Delta$-torsionless, then a module $M$ is $\Delta$-dense. In particular if $\gen {}_RU\subseteq\Cogen U$, all left $R$-modules are $\Delta$-dense.

We generalize this result obtaining a characterization of the modules  $M$ which are $\Delta$-dense. Let us start with the following

\begin{lemma}\label{lemma:tecnicodenso}
Let $L$ be a right $S$-module and $\xymatrix{L\ar[r]^f&\Delta M}$ a morphism. Then there exists $\alpha\in\Hom_R(M,\Delta L)$ such that $\Delta\alpha\circ \delta_L=f$.
\end{lemma}
\begin{proof}
We have the commutative diagram
\[\xymatrix{L\ar[r]^f\ar[d]_{\delta_L}&\Delta M\ar[d]^{\delta_{\Delta M}}\\
\Delta^2 L\ar[r]^{\Delta^2 f}&\Delta^3 M}\]
Let us consider $\alpha:=\Delta f\circ\delta_M:M\to\Delta L$. Then
\[\Delta\alpha\circ\delta_L=\Delta(\delta_M)\circ\Delta^2 f\circ\delta_L=
\Delta(\delta_M)\circ \delta_{\Delta M}\circ f=f.\]
\end{proof}

\begin{theorem}\label{thm:densita}
Let $_RU_S$ be a 1-cotilting bimodule with $S=\End{}_RU$. A left $R$-module ${}_RM$ is $\Delta$-dense if and only if for any $k\in\mathbb N$ and $\alpha:M\to {}_RU^k$, the $\Coker \alpha$ is $\D$-torsionless.
\end{theorem}
\begin{proof}
By Definition~\ref{defi:denseDeltalc}, a left $R$-module $M$ is $\Delta$-dense if and only if for all $k\in\mathbb N$ and all morphisms
\[\xymatrix{S^k\ar[r]^j&\Delta M}\quad \text{and}\quad\xymatrix{\Delta M\ar[r]^\phi&U}\]
there exists $m\in M$ such that $\phi\circ j=\delta_m\circ j$, or equivalently $\Delta j(\phi)=\Delta j(\delta_m)$, i.e., $\Im \Delta j=\Im(\Delta j\circ\delta_M)$. By Lemma~\ref{lemma:tecnicodenso} there exists $\alpha:M\to \Delta S^k=U^k$ such that $j=\Delta\alpha\circ \delta_{S^k}$. Then $\Im \Delta j=\Im(\Delta j\circ\delta_M)$ is equivalent to 
\[\Im \Delta(\delta_{S^k})\circ\Delta^2\alpha=
\Im \Delta(\delta_{S^k})\circ\Delta^2\alpha\circ\delta_M,
\]
which corresponds to $\Im \Delta^2\alpha=\Im \Delta^2\alpha\circ\delta_M$ being $\delta_{S^k}$ an isomorphism. We have the following commutative diagram
\[\xymatrix{M\ar[d]_{\delta_M}\ar[r]^\alpha&U^k\ar[d]_\cong^{\delta_{U^k}}\ar[r]&\Coker\alpha\ar[r]\ar[d]_{\eta^0_{\Coker\alpha}}&0\\
\Delta^2 M\ar[r]^{\Delta^2\alpha}&\Delta^2 U^k\ar[r]&H^0\R\Delta^2\Coker\alpha\ar[r]&0}\]
Then $\Im \Delta^2\alpha\circ\delta_M=\Im \Delta^2\alpha$ if and only if 
\[\Im \delta_{U^k}\circ \alpha=\delta_{U^k}(\Im\alpha)=\Im \Delta^2\alpha;\]
it is easy to verify by diagram chasing that this happens if and only if $\eta^0_{\Coker\alpha}$ is a monomorphism.
\end{proof}

\begin{corollary}
Let $_RU_S$ be a 1-cotilting bimodule with $S=\End{}_RU$. All left $R$-modules are $\Delta$-dense if and only if $\gen {}_RU\subseteq \D\Cogen U$.
\end{corollary}

\begin{corollary}\label{coro:noetheriano}
Let $_RU_S$ be a faithfully balanced 1-cotilting bimodule. If $\gen U=\pres U$ or $U$ is noetherian, then all modules are $\Delta$-dense and those in $\gen U$ are $\D$-reflexive. In particular a $\Delta$-torsionless module is $\D$-reflexive if and only if it is \emph{$\Delta$-linearly compact}.
\end{corollary}
\begin{proof}
Let us prove that $\gen U\subseteq \D\Cogen U$. Given $N\in\gen U$, consider the commutative diagram with exact rows
\[\xymatrix{&0\ar[r]&L\ar[d]_{\delta_L}\ar[r]& U^k\ar[d]^\cong_{\delta_{U^k}}\ar[r]& N\ar[d]_{\eta^0_N}\ar[r]&0\\
0\ar[r]&\Delta\Gamma N\ar[r]&\Delta^2 L\ar[r]& \Delta^2 U^k\ar[r]& H^0\R\Delta^2 N\ar[r]&0}\]
If $\gen U=\pres U$, we can assume $L\in\gen U$; if $U$ is noetherian, then $L$ is finitely generated. In both the cases $\delta_L$ is an isomorphism. Then $\Delta\Gamma N=0$ and $\eta^0_N$ is an isomorphism.
\end{proof}

We have the following closure properties for $\D$-reflexive $\Delta$-torsionless modules.
\begin{proposition}
Let $f:M\to N$ a morphism between $\Delta$-torsionless modules with $\Delta\Coker f=0$. If $M$ is $\Delta$-reflexive, then $\Ker f$ and $\Im f$ are $\Delta$-reflexive. Moreover $N$ is $\Delta$-reflexive if and only if $\Delta\Gamma^2\Coker f=0$; in such a case $\Coker f$ is $\D$-reflexive.
\end{proposition}
\begin{proof}
Let $K:=\Ker f$ and $C:=\Coker f$. Consider the exact sequence
\[\xymatrix{0\ar[r]& K\ar[r]& M\ar[r]^f\ar@{->>}[rd]& N\ar[r]& C\ar[r]&0\\ 
&&&I=\Im f\ar@{_(->}[u]}\]
By the diagram
\[\xymatrix{0\ar[r]& K\ar@{^(->}[d]^{\eta^0_K=\delta_K}\ar[r]& M\ar[d]^{\eta^0_M=\delta_M}_\cong\ar[r]& I\ar@{^(->}[d]^{\eta^0_I=\delta_I}\ar[r]&0\\ 
0\ar[r]& \Delta^2 K\ar[r]& \Delta^2 M\ar[r]& \Delta^2 I\ar[r]&0}\]
it follows easily that both $\delta_I$ and $\delta_K$ are isomorphisms.
Moreover, by the diagram
\[(*)\quad\xymatrix{&0\ar[r]& I\ar[d]_\cong\ar[r]& N\ar[d]\ar[r]& C\ar[d]\ar[r]&0\\ 
0\ar[r]&\Delta\Gamma C\ar[r]& \Delta^2 I\ar[r]& \Delta^2 N\ar[r]& \Gamma^2 C\ar[r]&0}\]
it follows that $\Delta\Gamma C=0$.
Applying the snake lemma to the diagram
\[\xymatrix{0\ar[r]&\Delta N\ar@{=}[r]\ar[d]&\Delta N\ar[r]\ar[d]_g&0\ar[d]\\
0\ar[r]&\Delta I \ar[r]&\Delta M\ar[r]&\Delta K\ar[r]&0}
\]
we get the exact sequence $0\to \Gamma C\to\Coker g\to \Delta K\to 0$.
We have the diagram
\[\xymatrix@-1pc{&&&0\ar[r]& \Gamma C\ar@{^(->}[d]^{\eta^0_{\Gamma C}}\ar[r]& \Coker g\ar[d]\ar[r]& \Delta K\ar[d]_\cong^{\eta^0_{\Delta K}=\delta_{\Delta K}}\ar[r]&0\\ 
0\ar[r]&\Delta\Gamma^2 C\ar[r]&\Delta\Gamma \Coker g\ar[r]&0\ar[r]& \Gamma^3 C\ar[r]& H^0\R\Delta^2\Coker g\ar[r]& \Delta^3 K\ar[r]&0}\]
where $\eta^0_{\Gamma C}$ is a monomorphism by Proposition~\ref{prop:Dtorsionless}, and $\delta_{\Delta K}$ is an isomorphism since $\Delta(\delta_K)\circ\delta_{\Delta K}=1_{\Delta K}$ and we have seen that $\delta_K$ is an isomorphism. Therefore $\eta^0_{\Coker g}$ is a monomorphism and $\Delta\Gamma^2 C\cong\Delta\Gamma\Coker g$.
By the diagram
\[\xymatrix{&0\ar[r]& \Delta N\ar@{^(->}[d]_{\delta_{\Delta N}}\ar[r]^g& \Delta M\ar[d]_\cong\ar[r]& \Coker g\ar@{^(->}[d]_{\eta^0_{\Coker g}}\ar[r]&0\\ 
0\ar[r]&\Delta\Gamma\Coker g\ar[r]& \Delta^3 N\ar[r]& \Delta^3 M\ar[r]& H^0\R\Delta^2\Coker g\ar[r]&0}
\]
we get that $\eta^0_{\Coker g}$ is an isomorphism and hence $\delta_{\Delta N}$ is an isomorphism if and only if $\Delta\Gamma\Coker g=0$. Then we get the thesis since $\delta_{\Delta N}$ is an isomorphism if and only if $\delta_N$ is an isomorphism, and $\Delta\Gamma\Coker g\cong \Delta\Gamma^2 C$.
\end{proof}

\subsection{$\D$-reflexive $\Delta$-torsion modules}\label{sec:Drifltorsion}

Given a $\Delta$-torsion module $N$, the first condition for the $\D$-reflexivity (see Proposition~\ref{prop:cottheorem}), i.e., $\Delta\Gamma N=0$, is not anymore automatically satisfied. 
It is not easy to understand in general when $\Delta\Gamma=0$. Here are some partial results:
\begin{lemma}\label{lemma:Bon}
Let $_RU_S$ be a $1$-cotilting bimodule and $M$ a left $R$-module. 
\begin{enumerate}
\item If $R=\End(U_S)$ and $M$ is finitely related, then
$\Delta\Gamma M=0$;
\item if $S=\End({}_RU)$ and the right $S$-module $\Gamma M$ is finitely generated, then $\Delta\Gamma M=0$;
\item if the torsion pair $(\Ker\Delta,\Ker\Gamma)$ is hereditary then $\Delta\Gamma=0$ on the whole module category.
\end{enumerate}
\end{lemma}
\begin{proof}
1. We have an exact sequence $0\to K\to P\to M\to 0$ with $P$ projective and $K$ finitely generated. Since $K$ is $\D$-reflexive by Lemma~\ref{lemma:fingengenerale}, 1a), we have the following diagram with exact rows
\[\xymatrix{&0\ar[r]&K\ar[d]_{\delta_{K}}^\cong\ar[r]&P\ar@{^(->}[d]_{\delta_{P}}\ar[r]&M\ar[d]\ar[r]&0\\
0=\Delta\Gamma P\ar[r]& \Delta\Gamma M\ar[r]&\Delta^2 K\ar[r]&\Delta^2 P\ar[r]&H^0\R\Delta^2 M\ar[r]&0}\]
from which easily one gets $\Delta\Gamma M=0$.\\
2. By the dual version of Bongartz Lemma (see \cite[Proposition~3.3.9]{GobTrl06}), denoted by $n$ the cardinality of a finite system of generators for the right $S$-module $\Gamma M$, there exists a short exact sequence
\[0\to U^n\to C\to M\to 0\]
for a suitable module $C$ in $\Ker\Gamma$.
Since $U^n$ is $\Delta$-reflexive, we have the following diagram with exact rows
\[\xymatrix{&0\ar[r]&U^n\ar[d]_{\delta_{U^n}}^\cong\ar[r]&C\ar@{^(->}[d]_{\delta_{C}}\ar[r]&M\ar[d]\ar[r]&0\\
0=\Delta\Gamma C\ar[r]& \Delta\Gamma M\ar[r]&\Delta^2 U^n\ar[r]&\Delta^2 C\ar[r]&H^0\R\Delta^2 M\ar[r]&0}\]
from which easily one gets $\Delta\Gamma M=0$.\\
3. See \cite[Theorem~1.7]{Man01}.
\end{proof}

\begin{definition}
Let $_RU_S$ be a $1$-cotilting bimodule. A module $N$ is 
\begin{itemize}
\item \textbf{$\ell$-orthogonal}, $\ell\geq 1$, if 
\[\Delta\Gamma N=0, ..., \Delta\Gamma^{\ell} N=0;\]
\item \textbf{$\infty$-orthogonal} if $N$ is $\ell$-orthogonal for each $\ell\in\mathbb N$.
\end{itemize}
\end{definition}

Any $\Delta$-torsionless module is clearly $\infty$-orthogonal. A $\D$-reflexive $\Delta$-torsion module is by definition $1$-orthogonal, but it is necessarily $\infty$-orthogonal:
\begin{proposition}\label{lemma:Grifl=rifl}
Let $_RU_S$ be a $1$-cotilting bimodule and $N$ a $\Delta$-torsion module. If $N$ is $\D$-reflexive, then $\Gamma^i N$ is $\Delta$-torsion and $\D$-reflexive for each $i\geq 1$.
\end{proposition}
\begin{proof}
Since $\Gamma^i N=\Gamma(\Gamma^{i-1} N)$, it is sufficient to prove the claim for $i=1$. By Proposition~\ref{prop:cottheorem} we have $\Delta\Gamma N=0$ and $\eta^0_N:N\to H^0\R\Delta^2 N$ is an isomorphism. By Corollary~\ref{cor:H0Delta}, $H^0\R\Delta^2 N=\Gamma^2 N$ and therefore $\Delta\Gamma(\Gamma N)=\Delta\Gamma^2 N=\Delta N=0$.
Since $\eta^0_N$ is an isomorphism, by Lemma~\ref{prop:key1}, (4),  also $\eta^0_{\Gamma N}$ is an isomorphism. 
By Proposition~\ref{prop:cottheorem} we conclude that $\Gamma N$ is $\D$-reflexive.
\end{proof}

If $S=\End({}_RU)$ and it is noetherian, by Lemma~\ref{lemma:fingengenerale}, (1c), and Proposition~\ref{lemma:Grifl=rifl} all finitely generated right $S$-modules are $\infty$-orthogonal.

If $_RU_S$ is a $1$-cotilting bimodule with $R=\End (U_S)$, a $\Delta$-torsionless left $R$-module is the direct limit of its finitely generated submodules and these are $\D$-reflexive by Proposition~\ref{lemma:fingengenerale}, (1a). Something similar happens to the $\Delta$-torsion modules:
\begin{lemma}\label{lemma:rejectsfg}
Let $_RU_S$ be a $1$-cotilting bimodule. Any $\Delta$-torsion module $N$ is the direct limit of the rejects of its finitely generated submodules. If $R=\End (U_S)$ and
$N$ is a $\D$-torsionless $1$-orthogonal left $R$-module, then the rejects of its finitely generated submodules are $\D$-reflexive.
\end{lemma}
\begin{proof}
Let $\{L_\lambda:\lambda\in \Lambda\}$ be the family of all finitely generated submodules of $N$. Applying the direct limit to the short exact sequence
\[0\to \Rej_U L_\lambda\to L_\lambda\to L_\lambda/\Rej_U L_\lambda\to 0\]
we get
\[0\to\varinjlim \Rej_U L_\lambda\to\varinjlim L_\lambda=N\to\varinjlim L_\lambda/\Rej_U L_\lambda\to 0\]
Since $N$ is $\Delta$-torsion, then $\Delta (\varinjlim L_\lambda/\Rej_U L_\lambda)=0$. Since $U_S$ is pure-injective \cite{Baz03}, $\Cogen U_S=\Ker\Gamma_S$ is closed under direct limits and therefore $\Gamma(\varinjlim L_\lambda/\Rej_U L_\lambda)=0$. Then we get $\varinjlim L_\lambda/\Rej_U L_\lambda=0$ and hence $N=\varinjlim \Rej_U L_\lambda$.
If $\Delta\Gamma N=0$, by Lemma~\ref{lemma:fingengenerale}, (1), and the closure of $\D\Cogen U$ with respect to submodules, the maps $\eta^0_{L_\lambda}$, $\lambda\in\Lambda$, are isomorphisms. Since $\Delta\Gamma N=0$ it is $\Delta\Gamma L_\lambda=0$ and hence the $L_\lambda$, $\lambda\in \Lambda$, are $\D$-reflexive.
Then, by Proposition~\ref{prop:iffrej}, also the $\Rej_UL_\lambda$'s are $\D$-reflexive.  Therefore $N$ is a direct limit of $\D$-reflexive $\Delta$-torsion submodules.
\end{proof}

Our aim now is to characterize the $\D$-reflexive $\Delta$-torsion modules. Compare the following with the notion of $\Delta$-density and of $\Delta$-linear compactness for $\Delta$-torsionless modules described in Definition~\ref{defi:denseDeltalc} and Remark~\ref{rem:Deltadense}.
\begin{definition}
Let $_RU_S$ be a 1-cotilting bimodule. If $N\in R\lMod$, then
\begin{enumerate}
\item $N$ is \textbf{$\Gamma$-linearly compact} if $N$ is $\Delta$-torsion and $\varprojlim p_\lambda:N\to \varprojlim L_\lambda$ is an epimorphism for every inverse system of epimorphisms $p_\lambda:N\to L_\lambda$ where the kernels $\Ker p_\lambda$ are $\Delta$-torsion and $2$-orthogonal;
\item $N$ is \textbf{$\Gamma$-dense} in $\Gamma^2 N$ ($\Gamma$-dense for short) if $N$ is $\Delta$-torsion and for each finitely generated submodule $F$ of $\Gamma N$,  denoted by $F\stackrel{i}{\hookrightarrow} \Gamma N$ the inclusion, we have $\Im(\Gamma i\circ\eta^0_N)=\Im\Gamma i$, i.e. the morphism
\[\Gamma i\circ\eta^0_N:N\to \Gamma F\text{ is epic}.\] 
\end{enumerate}
\end{definition}

The $\Gamma$-density is automatically satisfied on $2$-orthogonal $\Delta$-torsion modules:
\begin{proposition}\label{prop:Gdenso}
Let $_RU_S$ be a $1$-cotilting bimodule with $S=\End({}_RU)$. If $N$ is a $2$-orthogonal $\Delta$-torsion left $R$-module, then $N$ is $\Gamma$-dense.
\end{proposition}
\begin{proof}
Let $F$ be a finitely generated submodule of $\Gamma N$ and $i:F\to\Gamma N$ the inclusion.
Let us consider the composition
\[\xymatrix{
N\ar@/_1pc/[rr]_\phi\ar[r]^{\eta^0_N}&\Gamma^2 N\ar@{->>}[r]^-{\Gamma i}&\Gamma F=\Gamma\Rej_U F
}\]
We have to prove that $\phi$ is epic. Consider the exact sequence
\[\xymatrix{0\ar[r]&\Ker\phi\ar[r]& N\ar@{->>}[rd]\ar[rr]^\phi&&\Gamma \Rej_U F\ar[r]& \Coker\phi\ar[r]& 0\\
&&&I\ar@{^(->}[ru]}
\]
Since $\Delta\Gamma^2 N=0$, by Lemma~\ref{lemma:fingengenerale}, (1b), and Proposition~\ref{prop:iffrej}, $F$ and $\Rej_UF$ are $\D$-reflexive; in particular, by Proposition~\ref{prop:cottheorem}, $\Delta\Gamma \Rej_U F=0$ and hence $\Delta\Coker\phi=0$. Applying $\Delta$ we get the exact sequences
\[0\to \Delta I\to \Delta N=0\to\Delta\Ker\phi\to\Gamma I\to \Gamma N\text{ which implies }\Delta I=0\]
\[ 0=\Delta I\to\Gamma \Coker\phi\to \Gamma^2\Rej_U F\to\Gamma I\to 0\]
Therefore $\Gamma \Coker\phi$ is contained in $\Ker\Gamma\phi$. 
Since $\Gamma N$ is $\Delta$-torsion, we have the following commutative diagram
\[\xymatrix{\Gamma^2\Rej_U F=\Gamma^2 F\ar[rd]_{\Gamma^2 i}\ar[rr]^-{\Gamma\phi}&&\Gamma N\\
&\Gamma^3 N\ar[ru]_{\Gamma(\eta^0_N)}\\
\Rej_U F\ar[uu]^{\eta_{\Rej_U F}^0}_\cong\ar@{^(->}[r]^i&\Gamma N\ar[u]_{\eta^0_{\Gamma N}}
}\]
By Lemma~\ref{prop:key1}, (4),
$\Gamma(\eta^0_N)\circ\eta^0_{\Gamma N}=1_{\Gamma N}$; hence $\Gamma\phi$ is a monomorphism and $\Gamma\Coker\phi=0$. Since $_RU$ is cotilting and also $\Delta\Coker\phi=0$, we have $\Coker\phi=0$, i.e. $\phi$ is an epimorphism. 
\end{proof}

\begin{theorem}\label{teo:fond1}
Let $_RU_S$ be a $1$-cotilting bimodule. If a $\Delta$-torsion module $N$ is $\D$-reflexive, then it is $\Gamma$-linearly compact.
\end{theorem}
\begin{proof}
By Proposition~\ref{lemma:Grifl=rifl}, $N$ and $\Gamma^iN$ are $\Delta$-torsion and $\D$-reflexive for each $i\geq 0$.
Consider an inverse system of epimorphisms
\[0\to K_\lambda\to N\to N_\lambda\to 0\]
where the kernels $K_\lambda$ are $2$-orthogonal and $\Delta$-torsion modules.
Since $K_\lambda$, $N$ and $N_\lambda$ are $\Delta$-torsion and $\Delta\Gamma N=0$, by Proposition~\ref{prop:cottheorem} and Corollary~\ref{cor:H0Delta}, we have
\[H^{-1}\R\Delta^2 N=0,\ H^{0}\R\Delta^2 N_\lambda=\Gamma^2 N_\lambda,\ H^{0}\R\Delta^2 K_\lambda=\Gamma^2 K_\lambda,\ 
H^{0}\R\Delta^2 N=\Gamma^2 N\] 
and hence we have the following commutative diagram with exact rows
\[\xymatrix{
&0\ar[r]&K_\lambda\ar[r]\ar[d]_{\eta^0_{K_\lambda}}&N\ar[r]\ar[d]^\cong_{\eta^0_{N}}&N_\lambda\ar[r]\ar[d]^{\eta^0_{N_\lambda}}&0\\
0\ar[r]&\Delta\Gamma N_\lambda\ar[r]&\Gamma^2 K_\lambda\ar[r]&\Gamma^2 N\ar[r]&\Gamma^2 N_\lambda\ar[r]&0}\]
Since $\eta^0_N$ is an isomorphism,  we have that $\eta^0_{K_\lambda}$ is a monomorphism and $\eta^0_{N_\lambda}$ is an epimorphism.
Let us consider the short exact sequence
\[0\to K_\lambda\stackrel{\eta^0_{K_\lambda}}{\to} \Gamma^2 K_\lambda\to \Coker\eta^0_{K_\lambda}\to 0\]
Let us prove that $\Coker\eta^0_{K_\lambda}=0$ and hence $\eta^0_{K_\lambda}$ is an isomorphism.
We have the exact sequences
\[0\to\Delta \Coker\eta^0_{K_\lambda}\to \Delta\Gamma^2 K_\lambda=0, \]
which implies $\Delta \Coker\eta^0_{K_\lambda}=0$ and
\[\Delta K_\lambda=0\to\Gamma \Coker\eta^0_{K_\lambda}\to\Gamma^3 K_\lambda\stackrel{\Gamma(\eta^0_{K_\lambda})}{\to}\Gamma K_\lambda\to 0\]
By Proposition~\ref{prop:Dtorsionless}, $\eta^0_{\Gamma K_\lambda}$ is a monomorphism; therefore from the commutative diagram
\[\xymatrix{
\Gamma N\ar[d]^{\eta^0_{\Gamma N}}_\cong\ar[r]&\Gamma K_\lambda\ar[r]\ar@{^(->}[d]^{\eta^0_{\Gamma K_\lambda}}&0\\
{}\Gamma^3 N\ar[r]&\Gamma^3 K_\lambda\ar[r]&0
}
\]
we get that $\eta^0_{\Gamma K_\lambda}$ is an isomorphism. Then 
by Lemma~\ref{prop:key1}, (4), $\Gamma(\eta^0_{K_\lambda})$ is an isomorphism and hence $\Gamma \Coker\eta^0_{K_\lambda}=0$. Then  $\Coker\eta^0_{K_\lambda}=0$ and hence $\eta^0_{K_\lambda}$ is an isomorphism; then $\Delta\Gamma N_\lambda=0$ and $\eta^0_{N_\lambda}$ is an isomorphism.  By Proposition~\ref{prop:lincomp}, we conclude that $\varprojlim p_\lambda$ is an epimorphism.
\end{proof}

Let us prove now the converse of Theorem~\ref{teo:fond1}.
\begin{theorem}\label{teo:fondam2}
Let $_RU_S$ be a cotilting bimodule with $S=\End {}_RU$ and $N$ a $\D$-torsionless left $R$-module. If $N$ is $\Gamma$-linearly compact and $2$-orthogonal, then $N$ is $\D$-reflexive.
\end{theorem}
\begin{proof}
Since $\Delta\Gamma N=0$ and $N$ is $\D$-torsionless, we have to prove only that $\eta^0_N$ is an epimorphism.
By Lemma~\ref{lemma:rejectsfg}, the $S$-module $\Gamma N$ is a direct limit of the rejects of its finitely generated submodules $\{\Rej_U F_\lambda:\lambda\in\Lambda\}$, and these are $\D$-reflexive. 
Denoted by $i_\lambda$ and $j_\lambda$ the inclusions
\[ F_\lambda\stackrel{i_\lambda}{\hookrightarrow} \Gamma N\text{ and }\Rej_U F_\lambda\stackrel{j_\lambda}{\hookrightarrow}  F_\lambda,\]
let us consider the diagram
\[\xymatrix{
N\ar[d]^{\eta^0_N}&\Gamma \Rej_U F_\lambda\\
{}\Gamma^2 N\ar[r]^{\Gamma i_\lambda}&\Gamma F_\lambda\ar[u]^{\Gamma j_\lambda}_\cong
}
\]
Since $N$ is $\Gamma$-dense by Proposition~\ref{prop:Gdenso}, $\Gamma i_\lambda\circ \eta^0_N$ is an epimorphism and therefore also $\theta_\lambda:=\Gamma j_\lambda\circ\Gamma i_\lambda\circ \eta^0_N$ is an epimorphism. Let us consider the short exact sequence
\[\xymatrix{0\ar[r]& \Ker\theta_\lambda\ar[r]& N\ar[r]^-{\theta_\lambda}&\Gamma\Rej_U F_\lambda\ar[r]& 0}\]
Let us prove that $\Ker\theta_\lambda$ are $2$-orthogonal $\Delta$-torsion modules.
Applying $\Delta$ we get 
the following commutative diagram
\[\xymatrix{
{}\Delta N=0\ar[r]& \Delta \Ker\theta_\lambda\ar[r]&\Gamma^2 \Rej_U F_\lambda \ar[r]^-{\Gamma\theta_\lambda}\ar[rd]|-{\Gamma^2(i_\lambda\circ j_\lambda)}&\Gamma N \ar[r]&\Gamma\Ker\theta_\lambda\ar[r]&0
\\
&&&\Gamma^3 N\ar[u]_{\Gamma\eta^0_N}\\
&&\Rej_U F_\lambda\ar[uu]^{\eta^0_{\Rej_U F_\lambda}}_\cong\ar@{^(->}^{i_\lambda\circ j_\lambda}[r]&\Gamma N\ar[u]_{\eta^0_{\Gamma N}
}}
\]
Since $\Gamma(\eta^0_N)\circ\eta^0_{\Gamma N}=1_{\Gamma N}$ by Lemma~\ref{prop:key1}, (4), we have that $\Delta\Ker\theta_\lambda=0$ and hence $\Ker\theta_\lambda$ is $\Delta$-torsion. Applying $\Delta$ to the short exact sequence
\[0\to\Gamma^2\Rej_U F_\lambda\to \Gamma N\to \Gamma\Ker\theta_\lambda\to 0\]
we get the exact sequence $0\to \Delta\Gamma\Ker\theta_\lambda\to \Delta\Gamma N=0$ and hence $\Delta\Gamma\Ker\theta_\lambda=0$;
since $\Rej_U F_\lambda$ is $\D$-reflexive, by Proposition~\ref{lemma:Grifl=rifl} we have
\[\xymatrix{
0\ar[r]& \Delta\Gamma^2\Rej_U F_\lambda=0\ar[r]& \Gamma^2\Ker\theta_\lambda\ar[r]& \Gamma^2 N\ar[r]^-{\Gamma^2\theta_\lambda}&\Gamma^3 \Rej_U F_\lambda\ar[r]& 0}
\]
Applying again $\Delta$, we get the commutative diagram
\[\xymatrix{
{}\Delta\Gamma^2 N=0\ar[r]&\Delta\Gamma^2\Ker\theta_\lambda
\ar[r]&\Gamma^4\Rej_U F_\lambda\ar[r]^{\Gamma^3\theta_\lambda}&\Gamma^3 N\ar[r]& \Gamma^3\Ker\theta_\lambda\ar[r]&0\\
&&{}\Gamma^2\Rej_U F_\lambda\ar[u]^{\eta^0_{\Gamma^2\Rej_U F_\lambda}}_\cong\ar@{^(->}[r]^{\Gamma\theta_\lambda}&\Gamma N\ar@{^(->}[u]_{\eta^0_{\Gamma N}}
}\]
Then $\Gamma^3\theta_\lambda$ is mono and hence $\Gamma^2\Ker \theta_\lambda$ is $\Delta$-torsion: therefore $\Ker \theta_\lambda$ are  $2$-orthogonal $\Delta$-torsion modules.
Since $N$ is $\Gamma$-linearly compact, taking the inverse limit of  the commutative diagram
\[\xymatrix{N\ar[d]^{\eta^0_N}\ar@{->>}[rr]^{\theta_\lambda}&&\Gamma\Rej_U F_\lambda\\
\Gamma^2 N\ar[rr]^{\Gamma(i_\lambda\circ j_\lambda)}&&\Gamma\Rej_U F_\lambda\ar@{=}[u]
}
\]
we get the commutative diagram
\[\xymatrix{N\ar[d]^{\eta^0_N}\ar@{->>}[rr]^{\varprojlim\theta_\lambda}&&\varprojlim \Gamma\Rej_U F_\lambda\ar@{=}[d]\\
\Gamma^2 N\ar@{=}[d]\ar[rr]^{\varprojlim\Gamma(i_\lambda\circ j_\lambda)}
&&\varprojlim \Gamma\Rej_U F_\lambda\ar@{=}[d]\\
\Gamma^2 N\ar@{=}[d]\ar[rr]^{\Gamma\varinjlim(i_\lambda\circ j_\lambda)}
&&\Gamma\varinjlim \Rej_U F_\lambda\ar@{=}[d]\\
\Gamma^2 N\ar[rr]^{1_{\Gamma^2N}}&&\Gamma^2 N
}
\]
and hence $\eta^0_N$ is an epimorphism.
\end{proof}

\begin{corollary}\label{cor:fond}
Let $_RU_S$ be a 1-cotilting bimodule and $N$ a $\Delta$-torsion module. Then $N$ is $\D$-reflexive if and only if it is $3$-orthogonal and $\Gamma$-linearly compact.
\end{corollary}
\begin{proof}
If $\Delta\Gamma N=0=\Delta\Gamma^3 N$, then $N$ is a $\D$-torsionless module (see Proposition~\ref{prop:Dtorsionless}). Then we conclude by Proposition~\ref{lemma:Grifl=rifl} and Theorems~\ref{teo:fond1}, \ref{teo:fondam2}.
\end{proof}

We have the following closure properties for $\D$-reflexive $\Delta$-torsion modules.
\begin{proposition}
Let $_RU_S$ be a cotilting bimodule and $p:M\to N$ an epimorphism with $\Delta(\Ker p)=0$. If $M$ is a $\D$-reflexive $\Delta$-torsion module, then $N$ is $\D$-reflexive if and only if $N$ is 1-orthogonal and $\Ker p$ is 2-orthogonal; in such a case also $\Ker p$ is $\D$-reflexive.
\end{proposition}
\begin{proof}
Consider the commutative diagram with exact rows
\[\xymatrix@-1pc{&&&0\ar[r]&K=\Ker p\ar[d]^-{\eta^0_K}\ar[r]&M\ar[d]^-{\eta^0_M}_-\cong\ar[r]^p&N\ar[d]^-{\eta^0_N}\ar[r]&0\\
0\ar[r]&\Delta\Gamma K\ar[r]&\Delta\Gamma M=0\ar[r]&\Delta\Gamma N\ar[r]&\Gamma^2 K\ar[r]&\Gamma^2 M\ar[r]&\Gamma^2 N\ar[r]&0}\]
If $N$ is $\D$-reflexive, then $\Delta\Gamma N=0$ by definition and $K$ is $\D$-reflexive too. Thus $K$ and hence $\Rej_UK$ (see Proposition~\ref{prop:iffrej}) are $\D$-reflexive; then $\Delta\Gamma^2 K=\Delta\Gamma^2\Rej_U K=0$ by Proposition~\ref{lemma:Grifl=rifl}.
Let us see the converse. Since $\eta^0_M$ is an isomorphism, then $\eta^0_N$ is an epimorphism. Let us prove that $\Delta\Gamma^3 N=0$; thus, since $\Delta\Gamma N=0$ by hypothesis, we conclude by Proposition~\ref{prop:Dtorsionless} that $\eta^0_N$ is also a monomorphism and hence an isomorphism. Applying $\Delta$ to $\xymatrix{0\ar[r]&K\ar[r]&M\ar[r]&N\ar[r]&0}$, we get the exact sequence $\xymatrix{0\ar[r]&\Gamma N\ar[r]&\Gamma M\ar[r]&\Gamma K\ar[r]&0}$. From the following commutative diagram with exact rows
\[\xymatrix@-1pc{&&&0\ar[r]&\Gamma N\ar[d]^-{\eta^0_{\Gamma N}}\ar[r]&\Gamma M\ar[d]^-{\eta^0_{\Gamma M}}_-\cong\ar[r]&\Gamma K\ar[d]^-{\eta^0_{\Gamma K}}\ar[r]&0\\
0\ar[r]&\Delta\Gamma^2 N\ar[r]&\Delta\Gamma^2 M=0\ar[r]&\Delta\Gamma^2 K=0\ar[r]&\Gamma^3 N\ar[r]&\Gamma^3 M\ar[r]&\Gamma^3 K\ar[r]&0}\]
we obtain $\Delta\Gamma^2 N=0$. Moreover, by Proposition~\ref{prop:Dtorsionless}, the maps $\eta^0_{\Gamma K}$ and $\eta^0_{\Gamma N}$ are monomorphisms. Since $\eta^0_{\Gamma M}$ is an isomorphism, then $\eta^0_{\Gamma K}$ is an epimorphism and then both 
$\eta^0_{\Gamma K}$ and $\eta^0_{\Gamma N}$ are isomorphisms and hence $\Gamma N$ is a $\D$-reflexive $\Delta$-torsion module. By Proposition~\ref{lemma:Grifl=rifl} we conclude $\Delta\Gamma^3 N=0$.
\end{proof}

\section{Conclusions}
In this final part we collect the results proved in the previous sections in an unified version which permits us to characterize the $\D$-reflexive complexes.

\begin{definition}\label{def:lincompgen}
Let $_RU_S$ be a 1-cotilting bimodule.
A module $M$ is \textbf{$U$-linearly compact} if for each inverse system of epimorphisms $(M\stackrel{p_\lambda}{\to} M_\lambda)_{\lambda\in\Lambda}$ such that
\begin{enumerate}
\item $p_\lambda(\Rej_U M)=\Rej_U M_\lambda$,
\item $\Ker p_\lambda\cap \Rej_U M$ are $2$-orthogonal $\Delta$-torsion modules,
\end{enumerate}
the inverse limit $\varprojlim p_\lambda: M\to \varprojlim M_\lambda$ is an epimorphism.
\end{definition}

It is clear that if $_RU_S$ is a Morita bimodule, then the $U$-linearly compact modules are exactly the usual linearly compact modules, since $\Rej_U M=0$ for each module $M$.

\begin{proposition}\label{prop:compattezze}
A module $M$ is $U$-linearly compact if $\Rej_U M$ is $\Gamma$-linearly compact, and $M/\Rej_U M$ is $\Delta$-linearly compact. If $M$ is 2-orthogonal also the converse holds.
\end{proposition}
\begin{proof}
Consider an inverse system of epimorphisms $(M\stackrel{p_\lambda}{\to} M_\lambda)_{\lambda\in\Lambda}$ satisfying conditions (1), (2) in Definition~\ref {def:lincompgen}. Set $K_\lambda=\Ker p_\lambda$, we have the following commutative diagram with exact rows and columns
\[\xymatrix{
&0\ar[d]&0\ar[d]&0\ar[d]\\
0\ar[r]&K_\lambda\cap \Rej_U M\ar[d]\ar[r]& K_\lambda\ar[d]\ar[r]&K_\lambda/K_\lambda\cap \Rej_U M\ar[r]\ar[d]& 0\\
0\ar[r]&\Rej_U M\ar[d]^{n_\lambda}\ar[r]& M\ar[d]^{p_\lambda}\ar[r]&M/\Rej_U M\ar[r]\ar[d]^{q_\lambda}& 0\\
0\ar[r]&\Rej_U M_\lambda\ar[r]\ar[d]& M_\lambda \ar[r]\ar[d]&M_\lambda/\Rej_U M_\lambda\ar[r]\ar[d]& 0\\
&0&0&0}\]
Since $K_\lambda\cap \Rej_U M$ are $2$-orthogonal $\Delta$-torsion modules and $\Rej_U M$ is $\Gamma$-linearly compact, then $\varprojlim n_\lambda$ is surjective. Since $M_\lambda/\Rej_U M_\lambda$ is $\Delta$-torsionless and $M/\Rej_U M$ is $\Delta$-linearly compact, then $\varprojlim q_\lambda$ is surjective.
Taking the inverse limit of this diagram we get the following commutative diagram with exact rows and columns
\[\xymatrix{
&0\ar[d]&0\ar[d]&0\ar[d]\\
0\ar[r]&\varprojlim K_\lambda\cap \Rej_U M\ar[d]\ar[r]& \varprojlim K_\lambda\ar[d]\ar[r]&\varprojlim  K_\lambda/K_\lambda\cap \Rej_U M\ar[d]&\\
0\ar[r]&\Rej_U M\ar[d]^{\varprojlim n_\lambda}\ar[r]^\alpha& M\ar[d]^{\varprojlim p_\lambda}\ar[r]^\beta&M/\Rej_U M\ar[r]\ar[d]^{\varprojlim q_\lambda}& 0\\
0\ar[r]&\varprojlim \Rej_U M_\lambda\ar[r]^\theta\ar[d]& \varprojlim M_\lambda \ar[r]^-\xi&\varprojlim M_\lambda/\Rej_U M_\lambda\ar[d]\\
&0&&0}\]
By the Snake Lemma we conclude that 
$\varprojlim p_\lambda$ is surjective and hence $M$ is $U$-linearly compact.\\
Let $M$ be a $U$-linearly compact module; let us prove that $\Rej_U M$ is $\Gamma$-linearly compact and $M/ \Rej_U M$ is $\Delta$-linearly compact. Consider an inverse system of epimorphisms $(\Rej_U M\stackrel{n_\lambda}{\to} N_\lambda)_{\lambda\in\Lambda}$ where the kernels $\Ker n_\lambda$ are $2$-orthogonal $\Delta$-torsion modules. Then we have the following commutative diagram
\[\xymatrix{
&0\ar[d]&0\ar[d]\\
0\ar[r]&\Ker n_\lambda\ar[d]\ar@{=}[r]& \Ker p_\lambda\ar[d]\\
0\ar[r]&\Rej_U M\ar[d]^{n_\lambda}\ar[r]^\iota& M\ar[d]^{p_\lambda}\ar[r]&M/\Rej_U M\ar[r]\ar@{=}[d]& 0\\
0\ar[r]&N_\lambda\ar[r]\ar[d]& M_\lambda \ar[r]\ar[d]&M/\Rej_U M\ar[r]& 0\\
&0&0
}\]
where $M_\lambda$ is the pushout of $n_\lambda$ and $\iota$. Since $N_\lambda$ is $\Delta$-torsion and $M/\Rej_U M$ is $\Delta$-torsionless we have
\[\Rej_U M_\lambda=N_\lambda=p_\lambda(\Rej_U M);\]
next $\Ker p_\lambda\cap \Rej_U M=\Ker p_\lambda=\Ker n_\lambda$ are $2$-orthogonal $\Delta$-torsion modules. Therefore, since $M$ is $U$-linearly compact, $\varprojlim p_\lambda$ is an epimorphism. We have the following commutative diagram
\[\xymatrix{
0\ar[r]&\Rej_U M\ar[d]^{\varprojlim n_\lambda}\ar[r]^\iota& M\ar[d]^{\varprojlim p_\lambda}\ar[r]&M/\Rej_U M\ar[r]\ar@{=}[d]& 0\\
0\ar[r]&\varprojlim N_\lambda\ar[r]^\theta& \varprojlim M_\lambda \ar[r]^-\xi\ar[d]&M/\Rej_U M\\
&&0
}\]
By the Snake Lemma, 
%
$\varprojlim n_\lambda$ is an epimorphism too, and hence $\Rej_U M$ is $\Gamma$-linearly compact.\\
Consider now an inverse system of epimorphisms $(M/\Rej_U M\stackrel{q_\lambda}{\to} Q_\lambda)_{\lambda\in\Lambda}$ where the $Q_\lambda$'s are $U$-torsionless. Then we have the following commutative diagram
\[\xymatrix{
&&0\ar[d]&0\ar[d]\\
0\ar[r]&\Rej_U M\ar@{=}[d]\ar[r]&K_\lambda\ar[d]\ar[r]& \Ker q_\lambda\ar[d]^j\ar[r]&0\\
0\ar[r]&\Rej_U M\ar[r]^\iota& M\ar[d]^{p_\lambda}\ar[r]^-\pi&M/\Rej_U M\ar[r]\ar[d]^{q_\lambda}& 0\\
&& Q_\lambda \ar@{=}[r]\ar[d]&Q_\lambda\ar[d]\\
&&0&0
}\]
where $K_\lambda$ is the pullback of $\pi$ and $j$. Since
\[\Ker p_\lambda\cap\Rej_U M=K_\lambda\cap\Rej_U M=\Rej_U M,\]
$\Delta\Gamma \Rej_U M=\Delta\Gamma M=0$, and $\Delta\Gamma^2 \Rej_U M=\Delta\Gamma^2 M=0$, the modules $\Ker p_\lambda\cap\Rej_U M$ are $2$-orthogonal and $\Delta$-torsion. Moreover $p_\lambda\Rej_U M=0=\Rej_U Q_\lambda$. Therefore, since $M$ is $U$-linearly compact, $\varprojlim p_\lambda$ is an epimorphism. From the following commutative diagram
\[\xymatrix{
M\ar[r]^-\pi\ar[d]^{\varprojlim p_\lambda}&M/\Rej_U M\ar[r]\ar[d]^{\varprojlim q_\lambda}&0\\
{}\varprojlim Q_\lambda\ar@{=}[r]\ar[d]&\varprojlim Q_\lambda\\
0
}\]
we have that $\varprojlim q_\lambda$ is an epimorphism too, and  hence $M/\Rej_U M$ is $\Delta$-linearly compact.
\end{proof}
The notion of $U$-linear compactness is closed under suitable epimorphisms.
\begin{lemma}
Let $_RU_S$ be a cotilting bimodule and $p:M\to N$ an epimorphism such that $p(\Rej_UM)=\Rej_U N$ and $\Ker p\cap\Rej M$ is a $2$-orthogonal $\Delta$-torsion module. Then, if $M$ is $U$-linearly compact, also $N$ is $U$-linearly compact.
\end{lemma}
\begin{proof}
Let $(N\stackrel{p_\lambda}{\to} N_\lambda)_\lambda$ be an inverse system of epimorphisms such that $\Ker p_\lambda\cap \Rej_UN$ are $2$-orthogonal $\Delta$-torsion modules and $p_\lambda(\Rej N)=\Rej N_\lambda$. Applying the snake lemma to the diagram
\[\xymatrix{&0\ar[r]\ar[d]&M\ar[d]_p\ar@{=}[r]&M\ar[d]^{p_\lambda\circ p}\ar[r]&0\\
0\ar[r]&\Ker p_\lambda\ar[r]&N\ar[r]^{p_\lambda}&N_\lambda\ar[r]&0
}\]
we get the short exact sequence
\[0\to\Ker p\to \Ker (p\circ p_\lambda)\to\Ker p_\lambda\to 0
\]
and hence, since $p(\Rej_UM)=\Rej_U N$,  the short exact sequence
\[0\to\Ker p\cap\Rej_UM\to \Ker (p_\lambda\circ p)\cap\Rej_UM\to\Ker p_\lambda\cap\Rej_UN\to 0\]
Then $\Ker (p_\lambda\circ p)\cap\Rej_UM$ are $2$-orthogonal $\Delta$-torsion modules, $p_\lambda\circ p(\Rej_UM)=\Rej_UN_\lambda$ and hence $\varprojlim (p_\lambda\circ p)$ is surjective. Since $\varprojlim (p_\lambda\circ p)=\varprojlim p_\lambda\circ p$, also $\varprojlim p_\lambda$ is an epimorphism.
\end{proof}

\begin{theorem}\label{teo:fine}
Let $_RU_S$ be a faithfully balanced 1-cotilting bimodule. Then a module is $\D$-reflexive if and only if it is $U$-linearly compact, $\Delta$-dense and $3$-orthogonal.
\end{theorem}
\begin{proof}
By Proposition~\ref{prop:iffrej}, a module $M$ is $\D$-reflexive if and only if both $\Rej_U M$ and $M/\Rej_U M$ are $\D$-reflexive. Assume  $M$ is $\D$-reflexive. By Propositions~\ref{prop:iffrej}, \ref{lemma:Grifl=rifl} and Theorems~\ref{teo:fond1}, \ref{teo:deltariflessivo} $\Gamma^i M=\Gamma^i\Rej_UM$ are $\Delta$-torsion for each $i\geq 1$, $\Rej_U M$ is $\Gamma$-linearly compact and $M/\Rej_U M$ is $\Delta$-linearly compact and $\Delta$-dense in $\Delta^2 (M/\Rej_U M)=\Delta^2 M$. Therefore, by Proposition~\ref{prop:compattezze}, $M$ is $U$-linearly compact and $\Delta$-dense. Conversely, assume $M$ is $U$-linearly compact, $\Delta$-dense and $3$-orthogonal. By Proposition~\ref{prop:compattezze}, $\Rej_U M$ is $\Gamma$-linearly compact and $M/\Rej_U M$ is $\Delta$-linearly compact and $\Delta$-dense. Then $M/\Rej_U M$ is $\D$-reflexive by Theorem~\ref{teo:deltariflessivo} and $\Rej_U M$ is $\D$-reflexive by Corollary~\ref{cor:fond}; therefore $M$ is $\D$-reflexive by Proposition~\ref{prop:iffrej}.
\end{proof}
Let us denote by \textbf{$_RU$-LCD3} (resp. \textbf{$U_S$-LCD3}) the subcategory of $U$-linearly compact, $\Delta$-dense and $3$-orthogonal left $R$- (reap. right $S$-) modules. Observe that by Theorem~\ref{teo:fine} and Proposition~\ref{lemma:Grifl=rifl} the modules in $U$-LCD3 are $\infty$-orthogonal. 

\begin{corollary}\label{cor:complessiriflessivi}
The total derived functors $\R\Hom_R(-, U)$ and $\R\Hom_S(-, U)$ associated to a 1-cotilting bimodule $_RU_S$ define a duality between the categories of complexes of left $R$- and right $S$- modules with cohomologies in $_RU$-LCD3 and $U_S$-LCD3:
\[\R\Hom_R(-,U):\D_{\text{$_RU$-LCD3}}(R)\dualita{}{} \D_{\text{$_SU$-LCD3}}(S):\R\Hom_S(-,U).\]
\end{corollary}

\end{document}